\documentclass[a4paper,leqno,10pt]{article}

\usepackage{epsf,graphicx,epsfig}
\usepackage{amscd}
\usepackage{amsmath,latexsym,amssymb,amsthm}
\usepackage[nospace,noadjust]{cite}
\usepackage{textcomp}
\usepackage{setspace,cite}
\usepackage{lscape,fancyhdr,fancybox}
\usepackage{stmaryrd}
\usepackage[all,cmtip]{xy}
\usepackage{tikz}
\usepackage{cancel}
\usetikzlibrary{shapes,arrows,decorations.markings}
\usepackage{graphicx}

\usepackage{color}
\usepackage{url}
\usepackage{enumerate}
\usepackage[mathscr]{euscript}

  \theoremstyle{plain}

\swapnumbers
    \newtheorem{thm}{Theorem}[section]
    \newtheorem{prop}[thm]{Proposition}

    \newtheorem{subsec}[thm]{}
\theoremstyle{definition}
    \newtheorem{defn}[thm]{Definition}
        \newtheorem{remark}[thm]{Remark}
    \newtheorem{exam}[thm]{Example}

\theoremstyle{remark}

\newcommand{\id}{\operatorname{id}}

\title{}
\author{}
\date{}
\usepackage{amssymb}

\usepackage{hyperref}
\hypersetup{
	colorlinks,
	citecolor=blue,
	filecolor=black,
	linkcolor=blue,
	urlcolor=black
}

\begin{document}
\title{Twisted Rota-Baxter families and NS-family algebras\footnote{2020 {\em Mathematics Subject Classification.} 16T25, 17B38, 16W99, 17A99.} \footnote{{\em Keywords.} (Twisted) $\mathcal{O}$-operator family, Dendriform (family) algebra, Associative Yang-Baxter family, NS-(family) algebra, Cohomology and Deformation. }}

\author{Apurba Das\footnote{Indian Institute of Technology, Kharagpur 721302, West Bengal, India. Email: apurbadas348@gmail.com}
}






\maketitle

\noindent

\thispagestyle{empty}

{\bf Abstract.} Family algebraic structures indexed by a semigroup first appeared in the algebraic aspects of renormalizations in quantum field theory. The concept of the Rota-Baxter family and its relation with (tri)dendriform family algebras have been recently discovered. In this paper, we first consider a notion of $\mathcal{O}$-operator family as a generalization of the Rota-Baxter family and define two variations of associative Yang-Baxter family that produce $\mathcal{O}$-operator families. Given a Hochschild $2$-cocycle on the underlying algebra, we also define a notion of twisted $\mathcal{O}$-operator family (in particular twisted Rota-Baxter family). We also introduce and study NS-family algebras as the underlying structure of twisted $\mathcal{O}$-operator families. 
Finally, we define suitable cohomology of twisted $\mathcal{O}$-operator families and NS-family algebras (in particular cohomology of Rota-Baxter families and dendriform family algebras) that govern their deformations.



\tableofcontents

\vspace{0.2cm}

\section{Introduction}

\subsection{Rota-Baxter operators, (twisted) $\mathcal{O}$-operators and dendriform algebras} Rota-Baxter operators (often called Baxter operators) are the algebraic abstraction of the integral operator. They first appeared in the work of G. Baxter in the fluctuation theory of probability \cite{B60}. Subsequently, such operators were studied by G.-C. Rota, P. Cartier and F. V. Atkinson, among others \cite{rota,cartier,atkinson}. In the last twenty years, Rota-Baxter operators pay much attention due to their connections with combinatorics, shuffle algebras, infinitesimal bialgebras, splitting of algebras and the algebraic formulations of the renormalization in quantum field theory \cite{guo-keigher,aguiar,aguiar-infhopf,aguiar-inf,BBGN13,PBG17,CK00}. In \cite{loday} J.-L. Loday introduced a notion of dendriform algebra as a triple $(D, \prec, \succ)$ consisting of a vector space $D$ together with bilinear maps $\prec, \succ : D \otimes D \rightarrow D$ satisfying
\begin{align*}
(x \prec y) \prec z =~& x \prec (y \prec z + y \succ z), \\
(x \succ y) \prec z =~& x \succ (y \prec z), \\
(x \prec y + x \succ y) \succ z =~& x \succ (y \succ z), ~ \text{ for } x, y, z \in D.
\end{align*} 
Aguiar first finds the connection between the Rota-Baxter operator and dendriform algebra structures \cite{aguiar}. Later, {K. Uchino} introduced a notion of $\mathcal{O}$-operator as a relative version of Rota-Baxter operator and noncommutative analogue of Poisson structures \cite{U08}. Let $A$ be an associative algebra and $M$ be an $A$-bimodule. An $\mathcal{O}$-operator on $M$ over the algebra $A$ (simply called an $\mathcal{O}$-operator) is a linear map $T: M \rightarrow A$ satisfying
\begin{align}\label{o-optr}
T(u) \cdot T(v) = T \big( T(u) \cdot v + u \cdot T(v)  \big), \text{ for } u, v \in M.
\end{align}
Here ~$\cdot$~ represents the multiplication on $A$ as well as left and right $A$-actions on $M$. A Rota-Baxter operator on the algebra $A$ is simply an $\mathcal{O}$-operator when $M = A$ the adjoint $A$-bimodule. A solution of associative Yang-Baxter equation gives rise to an $\mathcal{O}$-operator. The result of Aguiar is also generalized to $\mathcal{O}$-operators. Namely, an $\mathcal{O}$-operator $T: M \rightarrow A$ induces a dendriform algebra structure on $M$ with the structure operations $u \prec v = u \cdot T(v)$ and $u \succ v = T(u) \cdot v$, for $u, v \in M$. See \cite{eb,guo-book} for more results on $\mathcal{O}$-operators and dendriform algebras.


In \cite{sev-wein} P. \v{S}evera and A. Weinstein first considered twisted Poisson structures in the study of Dirac geometry in suitable Courant algebroids. While $\mathcal{O}$-operators are the noncommutative analogue of Poisson structures, twisted $\mathcal{O}$-operators (first introduced by K. Uchino \cite{U08}) are the noncommutative analogue of twisted Poisson structures. Let $A$ be an associative algebra, $M$ be an $A$-bimodule and $H$ be a Hochschild $2$-cocycle on $A$ with values in $M$. In other words, $H: A^{\otimes 2} \rightarrow M$ is a bilinear map satisfying $\delta_\mathrm{Hoch} H = 0$, equivalently,
\begin{align}\label{2-co-id}
a \cdot H (b, c) - H ( a \cdot b, c)+ H (a, b \cdot c) - H (a, b) \cdot c =0, \text{ for } a, b, c \in A.
\end{align}
A linear map $T : M \rightarrow A$ is called an $H$-twisted $\mathcal{O}$-operator if for all $u, v \in M$,
\begin{align*}
T(u) \cdot T(v) = T \big( T(u) \cdot v + u \cdot T (v) + H (T(u), T(v))   \big), \text{ for } u, v \in M.
\end{align*}
On the other hand, NS-algebras was first introduced by P. Leroux in the study of Nijenhuis operators and dendriform-Nijenhuis bialgebras. Uchino observed that $H$-twisted $\mathcal{O}$-operators give rise to NS-algebras \cite{U08}. This generalizes the construction of dendriform algebras from $\mathcal{O}$-operators. See \cite{A3,An,ns-guo,ospel} for more details on NS-algebras.

\subsection{Rota-Baxter family and dendriform family algebras}
Algebraic structures often appear in `family versions', in which the usual operations are replaced by a family of operations indexed by a semigroup $\Omega$. The family version of algebraic structures first appeared in the work of K. Ebrahimi-Fard, J. Gracia-Bondia and F. Patras in the algebraic formulations of renormalization in quantum field theory \cite{fard-patras} (see also \cite{kre}). The notion of the Rota-Baxter family was introduced in \cite{fard-patras} and further studied by L. Foissy, X. Gao, L. Guo, D. Manchon and Y.Zhang \cite{foissy,guo09,zhang-free,zhang}. Let $A$ be an associative algebra. A Rota-Baxter family on $A$ is a collection $\{ R_\alpha : A \rightarrow A \}_{\alpha \in \Omega}$ of linear maps indexed by $\Omega$, satisfying
\begin{align*}
R_\alpha (a) \cdot R_\beta (b) = R_{\alpha \beta} \big(  R_\alpha (a) \cdot b + a \cdot R_\beta (b) \big), ~ \text{ for } a, b \in A.
\end{align*}
A Rota-Baxter family naturally induces a dendriform family algebra (family version of dendriform structure). It has been shown in \cite{zhang} that a Rota-Baxter family induces a Rota-Baxter operator on the tensor product with the semigroup algebra. The same phenomenons also hold for dendriform family algebras. See also \cite{wang} for some other results about the Rota-Baxter family. 

\subsection{Twisted $\mathcal{O}$-operator family and NS-family algebras}\label{subsec-o}
Like an $\mathcal{O}$-operator is a relative version of the Rota-Baxter operator, the notion of $\mathcal{O}$-operator family is a relative version of the Rota-Baxter family. There is a close relationship between $\mathcal{O}$-operator families and dendriform family algebras. See Propositions \ref{prop-o-dend} and \ref{prop-functor}. We define two variations of the family analogue of the associative Yang-Baxter equation and called them associative Yang-Baxter family of type-I and type-II, respectively (Definition \ref{aybf-1-2}). An associative Yang-Baxter family of type-I induces a Rota-Baxter family on the underlying algebra. On the other hand, a skew-symmetric associative Yang-Baxter family of type-II induces an $\mathcal{O}$-operator family. See Propositions \ref{aybf-1-prop} and \ref{aybf-2-prop}. Next, we introduce the notion of an $H$-twisted $\mathcal{O}$-operator family and find a characterization in terms of graphs. In particular, we obtain the notion of twisted Rota-Baxter family. A Reynolds family is an example of a twisted $\mathcal{O}$-operator family. We observed that a ($H$-twisted) $\mathcal{O}$-operator family induces an ordinary (twisted) $\mathcal{O}$-operator. We also introduce the notion of NS-family algebras as the family analogue of NS-algebras (Definition \ref{defn-ns-fam}). They generalize dendriform family algebras. An NS-family algebra induces an ordinary NS-algebra on the tensor product with the semigroup algebra (Theorem \ref{ns-fam-ns}). Finally, we show that an $H$-twisted $\mathcal{O}$-operator family induces an NS-family algebra structure (Proposition \ref{twisted-rota-ns}). In particular, the following diagram commutes 
\begin{align}\label{comm-diag}
\xymatrix{
\text{Twisted } \mathcal{O} \text{-operator family} \ar[d] \ar[r] & \text{NS-family algebra} \ar[d] \\
\text{Twisted } \mathcal{O} \text{-operator} \ar[r] & \text{NS-algebra}.
}
\end{align}

\subsection{Cohomology and deformations of family structures}
Formal deformation theory of algebraic structures began with the seminal work of Gerstenhaber for associative algebras \cite{gers}. Such deformation theory is governed by the Hochschild cohomology theory of associative algebras. Later, cohomology and deformation theory has been generalized to various algebraic structures, including Lie algebras and Leibniz algebras \cite{nij-ric,bala}. See \cite{bala1} for cohomology and deformations of algebras over any binary quadratic operads. Recently, cohomology and deformation theory of Rota-Baxter operators, (twisted) $\mathcal{O}$-operators and dendriform algebras have been extensively studied in \cite{Das,A3,A4,TBGS19}. In the present paper, we also define the cohomology theory of (twisted) $\mathcal{O}$-operator family. In particular, we get the cohomology theory associated with an $\mathcal{O}$-operator family and Rota-Baxter family. As an application of our cohomology, we study formal one-parameter deformations of (twisted) $\mathcal{O}$-operator family. On another side, we define the cohomology theory of NS-family algebras. In particular, we obtain the cohomology of dendriform family algebras. Finally, some remarks about the deformations of NS-family algebras are made.

\medskip

\medskip

The paper is organized as follows. In Section \ref{sec-o-op}, we define $\mathcal{O}$-operator family as a generalization of the Rota-Baxter family and introduce associative Yang-Baxter family of type-I and type-II. We define and study the notions of $H$-twisted $\mathcal{O}$-operator family and NS-family algebras in Section \ref{sec-3}. Finally, in Section \ref{sec-5}, we define and study cohomology of (twisted) $\mathcal{O}$-operator families and NS-family algebras. We end this paper with an appendix where some results of the present paper are dualized in the context of coassociative coalgebras.

All vector spaces, algebras, linear and multilinear maps, tensor products are over a field ${\bf k}$ of characteristic $0$. Throughout the paper, $\Omega$ is a semigroup whose elements are usually denoted by $\alpha, \beta, \gamma, \ldots$ . 

\section{$\mathcal{O}$-operator family and associative Yang-Baxter family}\label{sec-o-op}

In this section, we first introduce $\mathcal{O}$-operator family as a relative version of the Rota-Baxter family. We also find the relations between $\mathcal{O}$-operator families and dendriform family algebras. Next, we introduce the associative Yang-Baxter family (of type-I and type-II) that produces $\mathcal{O}$-operator families.

\subsection{$\mathcal{O}$-operator families} 

Let $A$ be an associative algebra. An $A$-bimodule consists of a vector space $M$ together with two bilinear operations $ l: A \otimes M \rightarrow M$, $(a,u) \mapsto a \cdot u$ and $ r: M \otimes A \rightarrow M$, $(u,a) \mapsto u \cdot a$ (called the left and right $A$-actions, respectively) satisfying 
$$ (a \cdot b) \cdot u =  a \cdot (b \cdot u),~~~~(a \cdot u ) \cdot b = a \cdot (u \cdot b),~~~~(u \cdot a) \cdot b = u \cdot (a \cdot b),~~\text{for all}~a,b \in A,~u \in M.$$
An $A$-bimodule as above may be denoted by $(M, l, r)$ or simply by $M$ when no confusion arises. It follows that $A$ is an $A$-bimodule with the left  and right $A$-actions given by the algebra multiplication.

Let $\Omega$ be a fixed semigroup (may not be commutative).

\begin{defn}
An {\bf $\mathcal{O}$-operator family} on $M$ over the algebra $A$ (or simply called an $\mathcal{O}$-operator family) is a collection $\lbrace T_\alpha : M \rightarrow A \rbrace_{\alpha \in \Omega}$ of linear maps satisfying 
$$T_\alpha(u) \cdot T_\beta(v) = T_{\alpha\beta}(T_\alpha(u) \cdot v + u \cdot T_\beta(v)),~~\text{for all}~u,v \in M~\text{and}~\alpha, \beta \in \Omega.$$
\end{defn} 

Let $A'$ be another associative algebra and $M'$ be an $A'$-bimodule. Let $\{ T'_\alpha : M \rightarrow A \}_{\alpha \in \Omega}$ an $\mathcal{O}$-operator family (on $M'$ over the algebra $A'$). A {\bf morphism} of $\mathcal{O}$-operator families from 
$\lbrace T_\alpha : M \rightarrow A \rbrace_{\alpha \in \Omega}$ to $\lbrace T'_\alpha : M' \rightarrow A' \rbrace_{\alpha \in \Omega}$ consists of a pair $(\phi, \psi)$ of an algebra homomorphism $\phi : A \rightarrow A'$ and a linear map $\psi : M \rightarrow M'$ satisfying
\begin{align*}
\psi (a \cdot u) = \phi (a) \cdot' \psi (u), ~~~~ \psi (u \cdot a) = \psi (u) \cdot' \phi (a) ~~ \text{ and } \phi \circ T_\alpha = T'_\alpha \circ \psi, ~\text{for } a \in A, u \in M \text{ and } \alpha \in \Omega.
\end{align*}
We denote the category of $\mathcal{O}$-operator families and morphisms between them by ${\bf Ooperf}$.

\begin{remark}
(1) Any Rota-Baxter family $\{ R_\alpha  : A \rightarrow A \}_{\alpha \in \Omega}$ on the algebra $A$ can be regarded as an $\mathcal{O}$-operator family with the underlying $A$-bimodule given by $A$.

(2) Let $T: M \rightarrow A$ be an $\mathcal{O}$-operator (see Equation (\ref{o-optr})). Then the constant family $\lbrace T_\alpha \rbrace_{\alpha \in \Omega}$, where $T_\alpha =T$ for all $\alpha \in \Omega$, is an $\mathcal{O}$-operator family.
\end{remark}


\begin{defn}\label{defn-der-fam}
Let $A$ be an associative algebra and $M$ be an $A$-bimodule. A collection $\lbrace D_\alpha : A \rightarrow M \rbrace_{\alpha \in \Omega}$ of linear maps is said to be a {\bf derivation family} on $A$ with values in the $A$-bimodule $M$ if they satisfy
\begin{equation}\label{1}
D_{\alpha\beta}(a \cdot b)= D_{\alpha}(a) \cdot b + a \cdot D_{\beta}(b),~\text{for all}~a,b \in A~\text{and}~\alpha, \beta \in \Omega.
\end{equation}

A derivation family $\lbrace D_\alpha \rbrace_{\alpha \in \Omega}$ is said to be invertible if $\dim A= \dim M$ and each $D_\alpha : A \rightarrow M$ is invertible.
\end{defn}

\begin{prop}
Let $\lbrace D_\alpha \rbrace_{\alpha \in \Omega}$ be an invertible derivation family on $A$ with values in the $A$-bimodule $M$. Then the collection $\lbrace D^{-1}_\alpha \rbrace_{\alpha \in \Omega}$ is an $\mathcal{O}$-operator family.
\end{prop}

\begin{proof}
As $\lbrace D_\alpha \rbrace_{\alpha \in \Omega}$ is an invertible family, we may substitute $a= D_\alpha^{-1}(u)$ and $b=D_\beta ^{-1}(v)$ in the identity (\ref{1}). Then we get $D_{\alpha \beta}  (D_\alpha ^{-1}(u) \cdot D_\beta ^{-1}(v))= u \cdot D_\beta ^{-1}(v) + D_\alpha ^{-1}(u) \cdot v$, which is equivalent to 
\begin{align*}
D_\alpha^{-1}(u) \cdot D_\beta ^{-1}(v) = D_{\alpha \beta} ^{-1} (D_\alpha ^{-1}(u) \cdot v + u \cdot D_\beta ^{-1}(v)).
\end{align*}
This shows that the collection $\lbrace D_\alpha ^{-1} \rbrace_{\alpha \in \Omega}$ is an $\mathcal{O}$-operator family.
\end{proof}

Given an associative algebra $A$ and an $A$-bimodule $M$, the direct sum $A \oplus M$ carries an associative algebra structure with the product given by
\begin{align*}
(a, u) \star (b, v) = (a \cdot b, a \cdot v + u \cdot b),~ \text{ for } (a,u), (b,v) \in A \oplus M.
\end{align*}
This is called the {\bf semidirect product} algebra, denoted by $A \ltimes M$. In the following result, we show that an $\mathcal{O}$-operator family is equivalent to a Rota-Baxter family on the semidirect product.

\begin{prop}\label{lift-rbf}
A collection $\lbrace T_\alpha : M \rightarrow A \rbrace_{\alpha \in \Omega}$ of linear maps is an $\mathcal{O}$-operator family if and only if the collection $\lbrace \widehat{T}_\alpha : A \oplus M \rightarrow A \oplus M \rbrace_{\alpha \in \Omega}$ is a Rota-Baxter family on the semidirect product $A \ltimes M$, where $$\widehat{T}_\alpha(a,u) = (T_\alpha(u),0),~\text{for}~(a,u) \in A \oplus M \text{ and } \alpha \in \Omega.$$
\end{prop}

\begin{proof}
For any $(a,u), (b,v) \in A \oplus M$ and $\alpha, \beta \in \Omega$, we have 
\begin{align}\label{2}
\widehat{T}_\alpha(a,u) \star \widehat{T}_\beta(b,v) = (T_\alpha(u), 0) \star (T_\beta(u), 0) = \big( T_\alpha(u) \cdot T_\beta(v),0 \big).
\end{align}
On the other hand,
\begin{align}\label{3}
\widehat{T}_{\alpha \beta} (\widehat{T}_\alpha(a,u) \star (b,v) + (a,u) \star \widehat{T}_\beta (b,v)) &= \widehat{T}_{\alpha \beta} ((T_\alpha(u),0) \star (b,v) + (a,u) \star (T_\beta(v),0)) \nonumber \\
&= \widehat{T}_{\alpha \beta} \big( (T_\alpha(u) \cdot b,~ T_\alpha (u) \cdot v) + (a \cdot T_\beta(v),~ u \cdot T_\beta(v)) \big) \nonumber \\
&= \big( T_{\alpha \beta} (T_\alpha(u) \cdot v + u \cdot T_\beta(v)),0 \big).
\end{align}
Hence the result follows from (\ref{2}) and (\ref{3}).
\end{proof}


Let $A$ be an associative algebra. Consider the tensor product algebra $A \otimes {\bf k} \Omega$ with the product 
\begin{align}\label{a-tensor}
(a \otimes \alpha)   \bullet (b \otimes \beta) = a \cdot b \otimes \alpha \beta.
\end{align}
Moreover, if $M$ is an $A$-bimodule then $M \otimes {\bf k} \Omega$ can be given an $(A \otimes {\bf k} \Omega)$-bimodule structure with the left and right actions (both of them are denoted by the same notation as $\bullet$ )
\begin{align}\label{m-tensor}
(a \otimes \alpha) \bullet (u \otimes \beta) = a \cdot u \otimes \alpha \beta~~\text{and}~~(u \otimes \beta) \bullet (a \otimes \alpha) = u \cdot a \otimes \beta \alpha.
\end{align}

With these notations, we have the following result. The proof will be given later when we will consider the notion of twisted $\mathcal{O}$-operator family (see Proposition \ref{of-o}). 

\begin{prop}
Let $\lbrace T_\alpha : M \rightarrow A \rbrace_{\alpha \in \Omega}$ be an $\mathcal{O}$-operator family. Then the map $$T: M \otimes {\bf k} \Omega \rightarrow A \otimes {\bf k} \Omega~~\text{given by}~~T(u \otimes \alpha) = T_\alpha(u) \otimes \alpha$$ is an $\mathcal{O}$-operator on $M \otimes {\bf k} \Omega$ over the algebra $A \otimes {\bf k} \Omega$.
\end{prop}


\begin{defn}
Let $\{ T_\alpha : M \rightarrow A \}_{\alpha \in \Omega}$ and $\{ S_\alpha : M \rightarrow A \}_{\alpha \in \Omega}$ be two $\mathcal{O}$-operator families. They are said to be compatible if for any $\lambda , \eta \in {\bf k}$, the collection $\{ \lambda T_\alpha + \eta S_\alpha : M \rightarrow A \}_{\alpha \in \Omega}$ is an $\mathcal{O}$-operator family. Equivalently, for $u, v \in M$ and $\alpha, \beta \in \Omega$, the following holds
\begin{align}\label{compatibility}
T_{\alpha}(u) \cdot S_{\beta}(v) + S_{\alpha}(u) \cdot T_{\beta} (v) = T_{\alpha \beta} \big( S_{\alpha}(u) \cdot v + u \cdot S_{\beta}(v) \big) + S_{\alpha \beta} \big( T_{\alpha}(u) \cdot v + u \cdot T_{\beta}(v) \big).
\end{align}
\end{defn}

Next, we introduce the notion of the Nijenhuis family on associative algebra generalizing Nijenhuis operators. We show that an $\mathcal{O}$-operator family can be equivalently seen as a Nijenhuis family on the semidirect product.

\begin{defn}
A {\bf Nijenhuis family} on an associative algebra $A$ is given by a collection $\lbrace N_\alpha : A \rightarrow A \rbrace_{\alpha \in \Omega}$ of linear maps satisfying 
$$N_\alpha(a) \cdot N_\beta(b) = N_{\alpha \beta} \big( N_\alpha (a) \cdot b + a \cdot N_\beta(b) - N_{\alpha \beta}(a \cdot b) \big),~\text{for}~a,b \in A~\text{and}~\alpha, \beta \in \Omega.$$
\end{defn}

Any Nijenhuis operator $N: A \rightarrow A$ on the associative algebra $A$ can be viewed as a constant Nijenhuis family $\lbrace N_\alpha \rbrace_{\alpha \in \Omega}$, where $N_\alpha =N$ for all $\alpha \in \Omega$. 

\begin{prop}
Let $\lbrace N_\alpha : A \rightarrow A \rbrace_{\alpha \in \Omega}$ be a Nijenhuis family on an associative algebra $A$. Then the map $N: A \otimes {\bf k} \Omega \rightarrow A \otimes {\bf k} \Omega~\text{given by}~N(a \otimes \alpha) = N_\alpha(a) \otimes \alpha$ is a Nijenhuis operator on the algebra $A \otimes {\bf k} \Omega$.
\end{prop}

\begin{proof}
For any $a \otimes \alpha,~ b \otimes \beta \in A \otimes {\bf k} \Omega$, we have 
\begin{align}\label{4}
N(a \otimes \alpha) \bullet N(b \otimes \beta) = (N_\alpha(a) \otimes \alpha) \bullet (N_\beta(b) \otimes \beta)
                                       = N_\alpha(a) \cdot N_\beta(b) \otimes \alpha \beta.
\end{align}
On the other hand, 
\begin{align}\label{5}
&N\big( N(a \otimes \alpha) \bullet (b \otimes \beta) + (a \otimes \alpha) \bullet N(b \otimes \beta) - N((a \otimes \alpha) \bullet (b \otimes \beta)) \big) \nonumber \\
&= N \big( N_\alpha(a) \cdot b \otimes \alpha \beta + a \cdot N_\beta(b) \otimes \alpha \beta - N_{\alpha \beta}(a \cdot b) \otimes \alpha \beta \big) \nonumber \\
&= N_{\alpha \beta} \big( N_\alpha(a) \cdot b + a \cdot N_\beta(b) - N_{\alpha \beta}(a \cdot b) \big) \otimes \alpha \beta.
\end{align}
Since $\lbrace N_\alpha \rbrace_{\alpha \in \Omega}$ is a Nijenhuis family, it follows from (\ref{4}) and (\ref{5}) that $N$ is a Nijenhuis operator on $A \otimes {\bf k} \Omega$.
\end{proof}

Compatible $\mathcal{O}$-operator families gives rise to a Nijenhuis family under certain invertibility condition. More precisely, we have the following.

\begin{prop}
Let $\lbrace T_\alpha : M \rightarrow A \rbrace_{\alpha \in \Omega}$ and $\lbrace S_\alpha : M \rightarrow A \rbrace_{\alpha \in \Omega}$ be two compatible $\mathcal{O}$-operator families in which the family $\lbrace S_\alpha : M \rightarrow A \rbrace_{\alpha \in \Omega}$ (resp. $\lbrace T_\alpha : M \rightarrow A \rbrace_{\alpha \in \Omega}$) is invertible. Then the collection $\lbrace N_\alpha = T_{\alpha} \circ S_{\alpha}^{-1} : A \rightarrow A \rbrace_{\alpha \in \Omega}$ (resp. $\lbrace N_\alpha = S_{\alpha} \circ T_{\alpha}^{-1} : A \rightarrow A \rbrace_{\alpha \in \Omega}$) is a Nijenhuis family on $A$.
\end{prop}

\begin{proof}
We only prove the case in which the family $\lbrace S_\alpha : M \rightarrow A \rbrace_{\alpha \in \Omega}$ is invertible. The other case is similar. For any $a,b \in A$ and $\alpha, \beta \in \Omega$, we have
\begin{align*}
&N_{\alpha}(a) \cdot N_{\beta}(b) \\
                          &= T_{\alpha}(S_{\alpha}^{-1}(a)) \cdot T_{\beta}(S_{\beta}^{-1}(b))\\
                          &= T_{\alpha \beta} \big( (T_{\alpha}\circ S_{\alpha}^{-1}(a)) \cdot S^{-1}_{\beta}(b) + S_{\alpha}^{-1}(a) \cdot (T_{\beta} \circ S^{-1}_{\beta}(b)) \big)\\
                          &= T_{\alpha \beta} \circ S_{\alpha \beta}^{-1} ~ \big[ S_{\alpha \beta} \big(     (T_{\alpha}\circ S_{\alpha}^{-1}(a)) \cdot S^{-1}_{\beta}(b) + S_{\alpha}^{-1}(a) \cdot (T_{\beta} \circ S^{-1}_{\beta}(b))  \big)    \big]\\
                          &= T_{\alpha \beta} \circ S_{\alpha \beta}^{-1} ~ \big[ (T_{\alpha} \circ S_{\alpha}^{-1})(a) \cdot b ~+~ a \cdot (T_{\beta} \circ S_{\beta}^{-1})(b) ~-~T_{\alpha \beta} \big( a \cdot S_{\beta}^{-1}(b) + S^{-1}_{\alpha} (a) \cdot b \big) \big]  \quad (\text{by}~(\ref{compatibility}))  \\
                          &= T_{\alpha \beta} \circ S_{\alpha \beta}^{-1} ~ \big[ (T_{\alpha} \circ S_{\alpha}^{-1})(a) \cdot b ~+~ a \cdot (T_{\beta} \circ S_{\beta}^{-1})(b) -(T_{\alpha \beta} \circ S^{-1}_{\alpha \beta})(a \cdot b)) \big] \\
                &     \qquad \qquad   \qquad \qquad   \qquad \qquad    (\text{since}~\lbrace S_{\alpha}^{-1} \rbrace_{\alpha \in \Omega}~~\text{is a derivation family}) \\
                          &= N_{\alpha \beta} \big( N_{\alpha}(a) \cdot b + a \cdot N_{\beta}(b) -N_{\alpha \beta}(a \cdot b) \big).
\end{align*}
Hence the proof.
\end{proof}

The following result is similar to Proposition \ref{lift-rbf}, thus we omit the proof.

\begin{prop}\label{lift-nf}
Let $A$ be an associative algebra and $M$ an $A$-bimodule. A collection $\lbrace T_\alpha : M \rightarrow A \rbrace_{\alpha \in \Omega}$ of linear maps is an $\mathcal{O}$-operator family if and only if
the collection $\{ \widehat{T}_\alpha : A \oplus M \rightarrow A \oplus M \}_{\alpha \in \Omega}$ is a Nijenhuis family on the semidirect product algebra $A \ltimes M$.
\end{prop}



It follows from Proposition \ref{lift-rbf} and Proposition \ref{lift-nf} that an $\mathcal{O}$-operator family lifts to both Rota-Baxter family and Nijenhuis family on the semidirect product.

\begin{defn} \cite{zhang-free}
A {\bf dendriform family algebra} is a vector space $D$ together with a family of bilinear operations $\lbrace \prec_\alpha, \succ_\alpha : D \otimes D \rightarrow D \rbrace_{\alpha \in \Omega }$ satisfying 
\begin{align*}
(x \prec_\alpha y) \prec_\beta z                              &= x \prec_{\alpha \beta} (y \prec_\beta z + y \succ_\alpha z)\\
(x \succ_\alpha y) \prec_\beta z                              &= x \succ_\alpha (y \prec_\beta z),\\
(x \prec_\beta y + x \succ_\alpha y) \succ_{\alpha \beta} z   &= x \succ_\alpha (y \succ_\beta z), \text{ for } x,y,z \in D \text{ and } \alpha, \beta \in \Omega.
\end{align*}
\end{defn}

\begin{remark}
Any dendriform algebra $(D, \prec, \succ)$ can be considered as a constant dendriform family algebra $(D, \lbrace \prec_\alpha, \succ_\alpha  \rbrace_{\alpha \in \Omega })$, where $\prec_\alpha = \prec$ and $\succ_\alpha = \succ$ for all $\alpha \in \Omega$.
\end{remark}

\begin{defn}
Let $(D, \lbrace \prec_\alpha, \succ_\alpha \rbrace_{\alpha \in \Omega})$ and $(D^{'}, \lbrace \prec^{'}_\alpha, \succ^{'}_\alpha \rbrace_{\alpha \in \Omega})$ be two dendriform family algebras. A morphism between them is a linear map $f : D \rightarrow D^{'}$ satisfying 
$$f(x \prec_\alpha y) = f(x) \prec^{'}_\alpha f(y)~~\text{and}~~f(x \succ_\alpha y) = f(x) \succ^{'}_\alpha f(y), \text{ for } x,y \in D,~ \alpha \in \Omega.$$
Let {\bf Dendf} denote the category of dendriform family algebras and morphisms between them.
\end{defn}

The following result finds the connection between dendriform family algebras and ordinary dendriform algebras \cite{zhang}.

\begin{prop} \label{see}
Let $(D, \lbrace \prec_\alpha, \succ_\alpha \rbrace_{\alpha \in \Omega})$ be a dendriform family algebra. Then the vector space $D \otimes {\bf k} \Omega$ with the bilinear operations
$$(x \otimes \alpha) \prec (y \otimes \beta) = (x \prec_\beta y ) \otimes \alpha \beta ~~~ \text{ and } ~~~ (x \otimes \alpha) \succ (y \otimes \beta) = (x \succ_\alpha y ) \otimes \alpha \beta$$ 
forms a dendriform algebra.
\end{prop}

It has been in \cite{zhang-free} that a Rota-Baxter family on an algebra induces a dendriform family structure on the underlying vector space. In the following result, we generalize this construction for $\mathcal{O}$-operator families.

\begin{prop}\label{prop-o-dend}
 Let $\lbrace T_\alpha: M \rightarrow A \rbrace_{\alpha \in \Omega}$ be an $\mathcal{O}$-operator family. Then $(M, \lbrace \prec_\alpha, \succ_\alpha \rbrace_{\alpha \in \Omega})$ is a dendriform family algebra (often denoted by $M_{ \{T_\alpha \} }$), where 
$$u \prec_\alpha v=u \cdot T_\alpha(v)~~\text{and}~~u \succ_\alpha v = T_\alpha(u) \cdot v, ~ \text{for } u, v \in M \text{ and } \alpha \in \Omega.$$
Moreover, if $(\phi, \psi)$ is a {morphism} of $\mathcal{O}$-operator families from $\lbrace T_\alpha : M \rightarrow A \rbrace_{\alpha \in \Omega}$ to $\lbrace T^{'}_\alpha : M^{'} \rightarrow A^{'} \rbrace_{\alpha \in \Omega}$, then $\psi: M \rightarrow M^{'}$ is a morphism between induced dendriform family algebras. In other words, there is a functor $\mathcal{F} : {\bf Ooperf} \rightarrow {\bf Dendf}$.
\end{prop}

%



Next we show that a dendriform family structure is always induced by an $\mathcal{O}$-operator family. Given a dendriform family algebra $(D, \lbrace \prec_\alpha, \succ_\alpha \rbrace_{\alpha \in \Omega})$, consider the dendriform algebra $(D \otimes {\bf k} \Omega , \prec, \succ )$ defined in Proposition \ref{see}. Hence $D \otimes {\bf k} \Omega$ carries an associative algebra structure with the product given by $(x \otimes \alpha) \odot (y \otimes \beta)= (x \prec_{\beta} y + x \succ_{\alpha} y) \otimes \alpha \beta$. We denote this associative algebra by $(D \otimes {\bf k} \Omega)_{\mathrm{Tot}}$. Moreover, the vector space $D$ can be given a $(D \otimes {\bf k} \Omega)_{\mathrm{Tot}}$-bimodule structure with the left and right actions 
$$(x \otimes \alpha) \cdot  y= x \succ_{\alpha} y~~\text{and}~~y \cdot (x \otimes \alpha)= y \prec_{\alpha} x,~\text{for}~(x \otimes \alpha) \in (D \otimes {\bf k} \Omega)_{\mathrm{Tot}}~\text{and}~y \in D.$$ With these notations, the collection $\lbrace \id_{\alpha}: D \rightarrow (D \otimes {\bf k} \Omega)_{\mathrm{Tot}} \rbrace_{\alpha \in \Omega}$ of maps defined by $\mathrm{Id}_\alpha(x) = x \otimes \alpha$, for $x \in D$, is an $\mathcal{O}$-operator family. Moreover, the induced dendriform family structure on $D$ coincides with the given one. The above construction gives rise to a functor $\mathcal{G}  : {\bf Dendf} \rightarrow {\bf Ooperf}.$

\begin{prop}\label{prop-functor}
The functor $\mathcal{G}$ is left adjoint to the functor $\mathcal{F}$. That is, for any dendriform family algebra $D = (D, \{ \prec_\alpha, \succ_\alpha \}_{\alpha \in \Omega}) \in {\bf Dendf}$ and an $\mathcal{O}$-operator family $\{ T_\alpha : M \rightarrow A \}_{\alpha \in \Omega} \in {\bf Ooperf}$, we have
\begin{align*}
\mathrm{Hom}_{\bf Ooperf} \big(  D \xrightarrow{ \{ \mathrm{Id}_\alpha \} } (D \otimes {\bf k} \Omega)_\mathrm{Tot},~ M \xrightarrow{ \{ T_\alpha \} } A \big) ~\cong ~ \mathrm{Hom}_{\bf Dendf} (D, M_{ \{ T_\alpha \} }).
\end{align*}
\end{prop}

\begin{proof}
Let $f \in \mathrm{Hom}_{\bf Dendf} (D, M_{ \{ T_\alpha \} }).$ We define a map $T^f : (D \otimes {\bf k} \Omega)_\mathrm{Tot} \rightarrow A$ by
\begin{align*}
T^f (x \otimes \alpha) := T_\alpha (f (x)),  \text{ for } x \otimes \alpha \in (D \otimes {\bf k} \Omega)_\mathrm{Tot}.
\end{align*}
It is easy to verify that $T^f$ is an associative algebra homomorphism. Further, we have
\begin{align*}
f ((x \otimes \alpha) \cdot y) = T^f (x \otimes \alpha) \cdot f(y) ~~~ \text{ and } ~~~ f (y \cdot (x \otimes \alpha)) = f(y) \cdot T^f (x \otimes \alpha),
\end{align*}
for $x \otimes \alpha \in (D \otimes {\bf k} \Omega)_\mathrm{Tot}$ and $y \in D.$ Moreover, $T^f \circ \mathrm{Id}_\alpha = T_\alpha \circ f$, for all $\alpha \in \Omega$. This implies that the pair 
\begin{align*}
(T^f, f) \in \mathrm{Hom}_{\bf Ooperf} \big(  D \xrightarrow{ \{ \mathrm{Id}_\alpha \} } (D \otimes {\bf k} \Omega)_\mathrm{Tot},~ M \xrightarrow{ \{ T_\alpha \} } A \big).
\end{align*}
Conversely, if $(\phi, \psi) \in \mathrm{Hom}_{\bf Ooperf} \big(  D \xrightarrow{ \{ \mathrm{Id}_\alpha \} } (D \otimes {\bf k} \Omega)_\mathrm{Tot},~ M \xrightarrow{ \{ T_\alpha \} } A \big)$, then $\psi \in  \mathrm{Hom}_{\bf Dendf} (D, M_{ \{ T_\alpha \} })$. The above two correspondences are inverses to each other. This completes the proof.
\end{proof}

\subsection{Associative Yang-Baxter family}

In \cite{aguiar-infhopf}, M. Aguiar introduced the notion of associative Yang-Baxter equation and find connections with Rota-Baxter operators and infinitesimal bialgebras. Here we define two variations of the associative Yang-Baxter equation in the family context and find their relations with $\mathcal{O}$-operator family.

Let $A$ be an associative algebra. Using Sweedler's notion, we write any element $r_\alpha \in A^{\otimes 2}$ as $r_\alpha = \sum r^{(1)}_\alpha \otimes r^{(2)}_\alpha$. From now onward, we shall omit the summation and simply write $r_\alpha = r^{(1)}_\alpha \otimes r^{(2)}_\alpha$ if no confusion arises. For any $r_\alpha = r^{(1)}_\alpha \otimes r^{(2)}_\alpha$, we define three elements 
$$r^{12}_\alpha = r^{(1)}_\alpha \otimes r^{(2)}_\alpha \otimes 1, \qquad r^{13}_\alpha = r^{(1)}_\alpha \otimes 1 \otimes r^{(2)}_\alpha ~~~ \text{ and } ~~~ r^{23}_\alpha = 1 \otimes r^{(1)}_\alpha \otimes r^{(2)}_\alpha$$ 
of $A \otimes A \otimes A$ if $A$ has identity (resp. $A^{+} \otimes A^{+} \otimes A^{+}$, where $A^{+}= A \oplus {\bf k}$, if $A$ has no identity). Note that the associative multiplication on $A$ induces an associative multiplication on the tensor product $A \otimes A \otimes A$ (resp. $A^{+} \otimes A^{+} \otimes A^{+}$) by componentwise multiplication. We shall use this multiplication to define associative Yang-Baxter families.

\begin{defn}\label{aybf-1-2}
A collection $\lbrace r_\alpha \in A^{\otimes 2} \rbrace_{\alpha \in \Omega}$ of elements of $A^{\otimes 2}$ is said to be an
\begin{enumerate}
\item[(i)] {\bf associative Yang-Baxter family of type-I} (AYBF of type-I) if they satisfy 
\begin{equation}\label{8}
r^{13}_{\alpha \beta} r^{12}_\alpha - r^{12}_\alpha r^{23}_\beta + r^{23}_\beta r^{13}_{\alpha \beta} =0, \text{ for } \alpha, \beta \in \Omega,
\end{equation}

\item[(ii)] {\bf associative Yang-Baxter family of type-II} (AYBF of type-II) if they satisfy 
\begin{equation}\label{9}
r^{13}_\alpha r^{12}_\beta - r^{12}_{\alpha \beta} r^{23}_\alpha + r^{23}_\beta r^{13}_{\alpha \beta}=0, \text{ for } \alpha ,\beta \in \Omega.
\end{equation}
\end{enumerate}
\end{defn}

\begin{remark}
A constant family $\{ r_\alpha \in A^{\otimes 2} \}_{\alpha \in \Omega}$ with $r_\alpha = r$ for all $\alpha \in \Omega$, is an AYBF of type-I (or AYBF of type-II) if and only if $r$ is a solution of the associative Yang-Baxter equation (see \cite{aguiar-infhopf})
\begin{align*}
r^{13} r^{12} - r^{12} r^{23} + r^{23} r^{13}=0.
\end{align*}
\end{remark}

\begin{prop}\label{aybf-1-prop}
Let $\lbrace r_\alpha \rbrace_{\alpha \in \Omega}$ be an AYBF of type-I. Then the collection $\lbrace R_\alpha : A \rightarrow A \rbrace_{\alpha \in \Omega}$ of maps
\begin{equation}\label{10}
R_\alpha(a) = r^{(1)}_\alpha \cdot a \cdot  r^{(2)}_\alpha,~\text{for}~a \in A
\end{equation}
is a Rota-Baxter family.
\end{prop}

\begin{proof}
Note that the identity (\ref{8}) can be explicitly written as 
$$r^{(1)}_{\alpha \beta} \cdot r^{(1)}_\alpha \otimes r^{(2)}_\alpha \otimes r^{(2)}_{\alpha \beta} - r^{(1)}_\alpha \otimes r^{(2)}_\alpha \cdot r^{(1)}_\beta \otimes r^{(2)}_\beta + r^{(1)}_{\alpha \beta} \otimes r^{(1)}_\beta \otimes r^{(2)}_\beta \cdot r^{(2)}_{\alpha \beta} =0,~\text{for } \alpha, \beta \in \Omega.$$
Replacing the first tensor product by $a$ and the second tensor product by $b$ and using (\ref{10}), we get that the collection $\lbrace R_\alpha \rbrace_{\alpha \in \Omega}$ is a Rota-Baxter family.
\end{proof}

A collection $\lbrace r_\alpha \in A^{\otimes 2} \rbrace_{\alpha \in \Omega}$ of elements of $A^{\otimes 2}$ is said to be skew-symmetric if each $r_\alpha$ is skew-symmetric, i.e., 
$r^{(1)}_\alpha \otimes r^{(2)}_\alpha = - r^{(2)}_\alpha \otimes r^{(1)}_\alpha, \text{ for }\alpha \in \Omega.$
Equivalently, $r_\alpha = - \tau(r_\alpha)$, for $\alpha \in \Omega$, where $\tau: A^{\otimes 2} \rightarrow A^{\otimes 2}$ is the switching map $\tau(a \otimes b) = b \otimes a$.

Note that, for an associative algebra $A$, the dual space $A^{*}$ carries an $A$-bimodule structure (called the coadjoint $A$-bimodule) with the left and right $A$-actions given by 
$$(a \cdot f)(b)=f(b \cdot a)~~\text{and}~~(f \cdot a)(b)=f(a \cdot b), \text{ for}~f \in A^{*} \text{ and } a,b \in A.$$

\begin{prop}\label{aybf-2-prop}
Let $\lbrace r_\alpha \rbrace_{\alpha \in \Omega}$ be a skew-symmetric AYBF of type-II. Then the collection $$\lbrace T_\alpha : A^{*} \rightarrow A \rbrace_{\alpha \in \Omega}  ~\text{ defined by }~ T_\alpha(f) = f (r^{(2)}_\alpha ) r^{(1)}_\alpha, \text{ for}~f \in A^{*}$$ 
is an $\mathcal{O}$-operator family (on $A^*$ over the algebra $A$).
\end{prop} 

\begin{proof}
For any $f,g \in A^{*}$ and $\alpha, \beta \in \Omega$, we have 
\begin{equation}\label{11}
T_\alpha(f) \cdot T_\beta(g) = f(r^{(2)}_\alpha) g(r^{(2)}_\beta)~ r^{(1)}_\alpha \cdot r^{(1)}_\beta.
\end{equation}
On the other hand, 
\begin{align}\label{12}
T_{\alpha \beta}\big( T_\alpha (f) \cdot g \big) &= (T_\alpha(f) \cdot g)(r^{(2)}_{\alpha \beta}) ~ r^{(1)}_{\alpha \beta} \nonumber \\
                                 &= g \big( r^{(2)}_{\alpha \beta} \cdot T_\alpha(f) \big) ~r^{(1)}_{\alpha \beta} \nonumber \\
                                 &= f(r^{(2)}_\alpha)~ g(r^{(2)}_{\alpha \beta} \cdot r^{(1)}_\alpha) ~r^{(1)}_{\alpha \beta},
\end{align}
and 
\begin{align}\label{13}
T_{\alpha \beta}(f \cdot T_\beta(g)) &= - g(r^{(1)}_\beta) ~ T_{\alpha \beta}(f \cdot r^{(2)}_\beta) \qquad (\text{using skew-symmetry of } r_\beta) \nonumber \\
                               &= - g(r^{(1)}_\beta) ~(f \cdot r^{(2)}_\beta)(r^{(2)}_{\alpha \beta}) ~r^{(1)}_{\alpha \beta} \nonumber \\
                               &= - g(r^{(1)}_\beta)~ f(r^{(2)}_\beta \cdot r^{(2)}_{\alpha \beta})~ r^{(1)}_{\alpha \beta}.
\end{align}
Since $\lbrace r_\alpha \rbrace_{\alpha \in \Omega}$ is an AYBF type-II, we have from condition (\ref{9}) that 
\begin{equation}\label{14}
r^{(1)}_\alpha \cdot r^{(1)}_\beta \otimes r^{(2)}_\beta \otimes r^{(2)}_\alpha - r^{(1)}_{\alpha \beta} \otimes r^{(2)}_{\alpha \beta} \cdot r^{(1)}_\alpha \otimes r^{(2)}_\alpha + r^{(1)}_{\alpha \beta} \otimes r^{(1)}_\beta \otimes r^{(2)}_\beta \cdot r^{(2)}_{\alpha \beta} = 0,\text{ for } \alpha, \beta \in \Omega.
\end{equation}
We define a linear map $(g \otimes f)^\sharp : A \otimes A \otimes A \rightarrow A$ by $(g \otimes f)^\sharp (a \otimes b \otimes c) = g(b) f(c) a.$
We apply this linear map to the both sides of (\ref{14}) and get 
$$g(r^{(2)}_\beta) f(r^{(2)}_\alpha) ~r^{(1)}_\alpha \cdot r^{(1)}_\beta ~ - ~ g(r^{(2)}_{\alpha \beta} \cdot r^{(1)}_\alpha) f(r^{(2)}_\alpha) ~ r^{(1)}_{\alpha \beta} ~+~ g(r^{(1)}_\beta) f(r^{(2)}_\beta  \cdot r^{(2)}_{\alpha \beta})~ r^{(1)}_{\alpha \beta} =0.$$
Thus, it follows from (\ref{11}), (\ref{12}) and (\ref{13}) that the collection $\lbrace T_\alpha \rbrace_{\alpha \in \Omega}$ is an $\mathcal{O}$-operator family.
\end{proof}


\section{Twisted $\mathcal{O}$-operator families and NS-family algebras}\label{sec-3}
In this section, we introduce twisted $\mathcal{O}$-operator family as a generalization of $\mathcal{O}$-operator family modified by a Hochschild $2$-cocycle. We also introduce NS-family algebras as a generalization of dendriform family algebras. NS-family structures are the underlying structure of twisted $\mathcal{O}$-operator families.

\subsection{Twisted $\mathcal{O}$-operator families}
Let $A$ be an associative algebra and $M$ be an $A$-bimodule. Suppose $H$ is a (Hochschild) $2$-cocycle of $A$ with values in the $A$-bimodule $M$, i.e. $H : A^{\otimes 2} \rightarrow M$ is a bilinear map satisfying (\ref{2-co-id}).

\begin{defn}\label{defn-tttt}
An {\bf $H$-twisted $\mathcal{O}$-operator family} on $M$ over the algebra $A$ (or simply an $H$-twisted $\mathcal{O}$-operator family) is a collection $\{ T_\alpha : M \rightarrow A \}_{\alpha \in \Omega}$ of linear maps satisfying
\begin{align}\label{twisted-o-fam}
T_\alpha  (u) \cdot T_\beta (v) = T_{\alpha \beta} \big( T_\alpha (u) \cdot v + u \cdot T_\beta (v) + H (T_\alpha (u), T_\beta (v))   \big),
\end{align}
for all $u, v \in M$ and $\alpha, \beta \in \Omega$.

When $M = A$ with the adjoint $A$-bimodule structure, an $H$-twisted $\mathcal{O}$-operator family is called an {\bf $H$-twisted Rota-Baxter family}. Further, if $H = 0$, we recover the notion of Rota-Baxter family.
\end{defn}

\begin{defn}\label{tw-o-map}
Let $A'$ be another associative algebra, $M'$ be an $A'$-bimodule and $H' : {A'}^{\otimes 2} \rightarrow M'$ be a $2$-cocycle. Suppose $\{ T'_\alpha : M' \rightarrow A' \}_{ \alpha \in \Omega}$ is an $H'$-twisted $\mathcal{O}$-operator family.
A {\bf morphism} of twisted $\mathcal{O}$-operator families from $\{ T_\alpha : M \rightarrow A \}_{ \alpha \in \Omega}$ to $\{ T'_\alpha : M' \rightarrow A' \}_{ \alpha \in \Omega}$ consists of a pair $(\phi, \psi)$ of an algebra homomorphism $\phi : A \rightarrow A'$ and a linear map $\psi : M \rightarrow M'$ satisfying
\begin{align*}
\psi (a \cdot u) =&~ \phi (a) \cdot' \psi (u), ~~~~~ \psi (u \cdot a) = \psi (u) \cdot' \phi (a), ~~~~~ \phi \circ T_\alpha = T'_\alpha \circ \psi ~~~ \text{ and } \\
&\psi \circ H = H' \circ (\phi \otimes \phi), ~ \text{ for } a \in A, u \in M \text{ and } \alpha \in \Omega.
\end{align*}
\end{defn}

Let {\bf TwOoperf} be the category of twisted $\mathcal{O}$-operator families and morphisms between them. Note that the category {\bf Ooperf} of $\mathcal{O}$-operator families is a subcategory of {\bf TwOoperf}.

\begin{exam}
Any $\mathcal{O}$-operator family is an $H$-twisted $\mathcal{O}$-operator family, for $H=0$. Therefore, twisted $\mathcal{O}$-operator family is a generalized notion of $\mathcal{O}$-operator family.
\end{exam}

\begin{exam}
Let $A$ be an associative algebra. Then we have seen that the space $A \otimes {\bf k} \Omega$ carries an associative algebra structure with the multiplication given by (\ref{a-tensor}).
There is also an $(A \otimes {\bf k} \Omega)$-bimodule structure on $A$ with the left and right $(A \otimes {\bf k} \Omega)$-actions (denoted by the same notation)
\begin{align*}
(a \otimes \alpha) \cdot b = a \cdot b ~~ \text{ and } ~~ b \cdot (a \otimes \alpha) = b \cdot a,
\end{align*}
for $a \otimes \alpha \in A \otimes {\bf k} \Omega$ and $b \in A$. Then the map $H : (A \otimes {\bf k} \Omega) \otimes (A \otimes {\bf k} \Omega) \rightarrow A$ given by $H (a \otimes \alpha, b \otimes \beta) = - a \cdot b$ is a Hochschild $2$-cocycle. With these notations, the collection $\{ \mathrm{Id}_\alpha : A \rightarrow A \otimes {\bf k} \Omega \}_{\alpha \in \Omega}$ defined by $\mathrm{Id}_\alpha (a) = a \otimes \alpha$ for $a \in A$, is an $H$-twisted $\mathcal{O}$-operator family.
\end{exam}

The next example generalizes the previous one using the Nijenhuis family.

\begin{exam}
Let $A$ be an associative algebra and $\{ N_\alpha : A \rightarrow A \}_{\alpha \in \Omega}$ be a Nijenhuis family on it. Then $A \otimes {\bf k} \Omega$ carries an associative algebra structure with the multiplication given by
\begin{align*}
(a \otimes \alpha) \bullet_N (b \otimes \beta) = N_\alpha (a) \cdot b \otimes \alpha \beta + a \cdot N_\beta (b) \otimes \alpha \beta - N_{\alpha \beta} (a \cdot b) \otimes \alpha \beta.
\end{align*}
We denote this associative algebra simply by $(A \otimes {\bf k} \Omega)_{N}$. There is an $(A \otimes {\bf k} \Omega)_{N}$-bimodule structure on $A$ with the the left and right $(A \otimes {\bf k} \Omega)_{N}$-actions given by (denoted by the same notation)
\begin{align*}
(a \otimes \alpha ) \cdot b = N_\alpha (a) \cdot b ~~~ \text{ and } ~~~ b \cdot (a \otimes \alpha) = b \cdot N_\alpha (a),
\end{align*}
for $a \otimes \alpha \in (A \otimes {\bf k} \Omega)_{N}$ and $b \in A$. Moreover, the map $H : (A \otimes {\bf k} \Omega)_{N} \otimes (A \otimes {\bf k} \Omega)_{N} \rightarrow A$ given by $H (a \otimes \alpha, b \otimes \beta) = - N_{\alpha \beta} (a \cdot b)$ is a Hochschild $2$-cocycle. Then it is easy to verify that the collection $\{ \mathrm{Id}_\alpha : A \rightarrow (A \otimes {\bf k} \Omega)_{N} \}_{\alpha \in \Omega}$ of maps is an $H$-twisted $\mathcal{O}$-operator family.
\end{exam}

In the next result, we characterize an $H$-twisted $\mathcal{O}$-operator family by its graph. In particular, we get a characterization of $\mathcal{O}$-operator family. We first need the following.

\medskip

$\diamond$ Given a $2$-cocycle $H$, the direct sum $A \oplus M$ carries an associative algebra with the product $\star_H$ given by
\begin{align*}
(a,u) \star_H (b, v) = (a \cdot b ,~ a \cdot v + u \cdot b + H (a, b)),
\end{align*}
for $(a,u), (b, v) \in A \oplus M$. This is called the $H$-twisted semidirect product algebra, denoted by $A \ltimes_H M$. In particular, when $H =0$, we get the standard semidirect product algebra $A \ltimes M.$

\medskip

$\diamond$ Let $A$ be an associative algebra. A collection $\lbrace B_\alpha \rbrace _{\alpha \in \Omega}$ of subspaces of $A$ is said to be a {\bf subalgebra family} of $A$ if the associative multiplication satisfies $B_\alpha \cdot B_\beta \subset B_{\alpha \beta}$, for $\alpha, \beta \in \Omega$.

\medskip

\begin{prop}
Let $A$ be an associative algebra, $M$ be an $A$-bimodule and $H$ be a $2$-cocycle. Then a collection $\{ T_\alpha : M \rightarrow A \}_{\alpha \in \Omega}$ of linear maps is an $H$-twisted $\mathcal{O}$-operator family if and only if the collection of graphs $\{ Gr (T_\alpha) \}_{\alpha \in \Omega}$ is a subalgebra family of the $H$-twisted semidirect product algebra $A \ltimes_H M$.
\end{prop}

\begin{proof}
For any $u, v \in M$ and $\alpha, \beta \in \Omega$, consider the elements $(T_\alpha (u), u) \in Gr(T_\alpha)$ and $(T_\beta(v), v) \in Gr(T_\beta)$. Their product in $A \ltimes_H M$ is given by 
$$(T_\alpha(u), u) \star_H (T_\beta (v), v) = \big( T_\alpha(u) \cdot T_\beta(v),~ T_\alpha(u) \cdot v + u \cdot T_\beta(v) + H (T_\alpha (u), T_\beta (v)) \big).$$
This is in $Gr(T_{\alpha \beta})$ if and only if (\ref{twisted-o-fam}) holds. In other words, the collection $\lbrace T_\alpha \rbrace_{\alpha \in \Omega}$ is an $H$-twisted $\mathcal{O}$-operator family.
\end{proof}

In \cite{reyn} O. Reynolds considered a linear operator (known as Reynolds operator) on an algebra in the study of fluctuation theory in fluid dynamics. Reynolds operators are closely related to averaging operators. In \cite{U08} K. Uchino observed that Reynolds operators can be seen as twisted Rota-Baxter operators. In the following, we consider the notion of the Reynolds family and show that any Reynolds family can be realized as a twisted Rota-Baxter family.

\begin{defn}
Let $A$ be an associative algebra. A {\bf Reynolds family} on $A$ consists of a collection $\{ R_\alpha : A \rightarrow A \}_{\alpha \in \Omega}$ of linear maps satisfying
\begin{align}\label{rey-fam}
R_\alpha (a) \cdot R_\beta (b) = R_{\alpha \beta} \big(    R_\alpha (a) \cdot b + a \cdot R_\beta (b) - R_\alpha (a) \cdot R_\beta (b ) \big), \text{ for } a, b \in A \text{ and } \alpha, \beta \in \Omega.
\end{align}
\end{defn}

\begin{remark}
Let $A$ be an associative algebra. Let $\mu : A^{\otimes 2} \rightarrow A$ denotes the associative multiplication map on $A$, i.e., $\mu (a, b) = a \cdot b$, for $a, b \in A$. Then $\mu$ is a Hochschild $2$-cocycle of $A$ with values in the adjoint $A$-bimodule. With this notation, a Reynolds family $\{ R_\alpha : A \rightarrow A \}_{\alpha \in \Omega}$ on the algebra $A$ is simply a $(-\mu)$-twisted Rota-Baxter family.
\end{remark}

Ler $\{ R_\alpha : A \rightarrow A \}_{\alpha \in \Omega}$ be an invertible Reynolds family. Then the identity (\ref{rey-fam}) can be equivalently written as
\begin{align*}
R_{\alpha \beta}^{-1} (a \cdot b) = R_\alpha^{-1} (a) \cdot b ~+~ a \cdot R_\beta^{-1} (b) - a \cdot b,
\end{align*}
for $a, b \in A$, $\alpha, \beta \in \Omega$, which is further equivalent to
\begin{align*}
(R_{\alpha \beta}^{-1} - \mathrm{Id}) (a \cdot b) = (R_{\alpha }^{-1} - \mathrm{Id})(a) \cdot b ~+~ a \cdot (R_{ \beta}^{-1} - \mathrm{Id}) (b).
\end{align*}
This shows that the collection $\{ R_\alpha^{-1} - \mathrm{Id} \}_{\alpha \in \Omega}$ of maps is a derivation family on $A$ (with values in the adjoint $A$-bimodule). Conversely, let $\{ D_\alpha \}_{\alpha \in \Omega}$ be a derivation family on $A$ such that the linear map $D_\alpha + \mathrm{Id}$ is invertible, for all $\alpha \in \Omega$. Then $\{ (D_\alpha + \mathrm{Id})^{-1} \}_{\alpha \in \Omega} $ is an (invertible) Reynolds family. In general, if $D_\alpha + \mathrm{Id}$ is not invertible, then we have the following.

\begin{prop}
Let $A$ be an associative algebra and $\{ D_\alpha \}_{\alpha \in \Omega}$ be a derivation family on $A$. Suppose, for each $a \in A$ and $\alpha \in \Omega$, the infinite sum $\sum_{n=0}^\infty (-1)^n D_\alpha^n (a)$ converges to an element in $A$. Then 
\begin{align*}
\{ R_\alpha = \sum_{n=0}^\infty (-1)^n D_\alpha^n \}_{\alpha \in \Omega} ~\text{ is a Reynolds family.}
\end{align*}
\end{prop}

\begin{proof}
For any $a, b \in A$ and $\alpha, \beta \in \Omega$, we have
\begin{align}\label{rey-pr1}
R_\alpha (a) \cdot R_\beta (b) = \sum_{m, n \geq 0} (-1)^{m+n} ~ D^m_\alpha (a) \cdot D^n_\beta (b). 
\end{align}
Hence 
\begin{align}\label{rey-pr2}
R_{\alpha \beta} \big(  R_\alpha (a) \cdot R_\beta (b)  \big) =~& \sum_{m, n , k \geq 0} (-1)^{m+n+k} ~ D^k_{\alpha \beta} \big( D^m_\alpha (a) \cdot D^n_\beta (b)   \big) \nonumber \\
=~& \sum_{m, n, k \geq 0} ~ \sum_{i+j = k} (-1)^{m+n+k} ~ {k \choose i} ~D^{m+i}_\alpha (a) \cdot D^{n+j}_\beta (b) \nonumber \\
=~& \sum_{m,n,i, j \geq 0} (-1)^{m+n+i+j} ~ {i+j \choose i} ~D^{m+i}_\alpha (a) \cdot D^{n+j}_\beta (b) \nonumber \\
=~& \sum_{p, q \geq 0} (-1)^{p+q} ~\big(  \sum_{i=0}^p \sum_{j=0}^q {i+j \choose i}  \big) ~ D^p_\alpha (a) \cdot D^q_\beta (b).
\end{align}
On the other hand, we have
\begin{align}\label{rey-pr3}
R_{\alpha \beta} \big( R_\alpha (a) \cdot b   \big) =~& \sum_{m, k \geq 0} \sum_{i+j = k} (-1)^{m+k} ~ {k \choose i} ~ D^{m+i}_\alpha (a) \cdot D_\beta^j (b) \nonumber \\
=~& \sum_{l \geq 0} (-1)^l \sum_{p+j = l} \sum_{m+i= p} {i+j \choose i}~ D^p_\alpha (a) \cdot D^j_\beta (b) \nonumber \\
=~& \sum_{p, j \geq 0} (-1)^{p+j} \sum_{i=0}^p {i+j \choose i}~ D^p_\alpha (a) \cdot D^j_\beta (b).
\end{align}
Similarly,
\begin{align}\label{rey-pr4}
R_{\alpha \beta} \big(  a \cdot R_\beta (b) \big) = \sum_{q, i \geq 0} (-1)^{q+i} \sum_{j=0}^q {i+j \choose j} ~ D^i_\alpha (a) \cdot D^q_\beta (b).
\end{align}
With the values (\ref{rey-pr1})-(\ref{rey-pr4}), we have the identity (\ref{rey-fam}) holds if and only if the coefficients of $D^p_\alpha (a) \cdot D^q_\beta (b)$ in both sides are equal. In other words, $\{ R_\alpha \}_{\alpha \in \Omega}$ is a Reynolds family if and only if
\begin{align*}
(-1)^{p+q} = (-1)^{p+q} \big\{  \sum_{i=0}^p {i+q \choose i} ~+~ \sum_{j=0}^q {p+j \choose j} ~-~ \sum_{i=0}^p \sum_{j=0}^q {i+j \choose i} \big\},\text{ for all } p, q \geq 0. 
\end{align*}
The above identity can be easily verified by the induction on $k = p+q$ (see \cite{guo-pei}). Hence $\{ R_\alpha \}_{\alpha \in \Omega}$ is a Reynolds family.
\end{proof}

Let $A$ be an associative algebra, $M$ be an $A$-bimodule and $H$ be a $2$-cocycle. Consider the associative algebra $A \otimes {\bf k} \Omega$ and the $(A \otimes {\bf k} \Omega)$-bimodule $M \otimes {\bf k } \Omega$ given in (\ref{a-tensor}) and (\ref{m-tensor}). Then the map 
\begin{align*}
\widehat{H}: (A \otimes {\bf k} \Omega)^{\otimes 2} \rightarrow M \otimes {\bf k} \Omega, ~~ \widehat{H} \big( a \otimes \alpha, b \otimes \beta  \big) = H(a, b) \otimes \alpha \beta
\end{align*}
is a $2$-cocycle on $A \otimes {\bf k} \Omega$ with values in $M \otimes {\bf k} \Omega$. The next result shows that a twisted $\mathcal{O}$-operator family induces a twisted $\mathcal{O}$-operator.

\begin{prop}\label{of-o}
Let $\{ T_\alpha : M \rightarrow A \}_{\alpha \in \Omega}$ be an $H$-twisted $\mathcal{O}$-operator family. Then the map
\begin{align*}
T : M \otimes {\bf k} \Omega \rightarrow A \otimes {\bf k} \Omega ~ \text{ given by } ~ T (u \otimes \alpha) = T_\alpha (u) \otimes \alpha
\end{align*}
is an $\widehat{H}$-twisted $\mathcal{O}$-operator (on $M \otimes {\bf k} \Omega$ over the algebra $A \otimes {\bf k} \Omega$).
\end{prop}

\begin{proof}
For any $u \otimes \alpha, v \otimes \beta \in M \otimes {\bf k} \Omega$, we have
\begin{align*}
&T (u \otimes \alpha) \bullet T (v \otimes \beta) \\
&= (T_\alpha (u) \otimes \alpha) \bullet (T_\beta (v) \otimes \beta) \\
&= \big( T_\alpha (u) \cdot T_\beta (v) \big) \otimes \alpha \beta \\
&= T_{\alpha \beta } \big(  T_\alpha (u) \cdot v ~+~ u \cdot T_\beta (v)  ~+~ H (T_\alpha (u), T_\beta (v))  \big)   \otimes \alpha \beta  \\
&= T \big(  ( T_\alpha (u) \cdot v ~+~ u \cdot T_\beta (v)  ~+~ H (T_\alpha (u), T_\beta (v))  )   \otimes \alpha \beta \big) \\
&= T \big(  (T_\alpha (u) \otimes \alpha) \bullet (v \otimes \beta)  ~+~ (u \otimes \alpha) \bullet (T_\beta (v) \otimes \beta) ~+~ H (T_\alpha (u), T_\beta (v)) \otimes \alpha \beta \big) \\
&= T \big(  T(u \otimes \alpha) \bullet (v \otimes \beta) ~+~ (u \otimes \alpha) \bullet T(v \otimes \beta) ~+~ \widehat{H} (T (u \otimes \alpha), T(v \otimes \beta) )  \big).
\end{align*}
This shows that $T$ is an $\widehat{H}$-twisted $\mathcal{O}$-operator.
\end{proof}

\subsection{NS-family algebras}
In this subsection, we introduce NS-family algebras as the family analogue of NS-algebras. This notion of NS-family algebra generalizes dendriform family algebras. We observed that an NS-family algebra induces an ordinary NS-algebra. Finally, we find the relations between twisted $\mathcal{O}$-operator families and NS-family algebras. 

\begin{defn} \cite{leroux}
An {\bf NS-algebra} is a quadruple $(D, \prec, \succ, \curlyvee)$ consists of a vector space $D$ together with three bilinear operations $\prec, \succ, \curlyvee : D \otimes D \rightarrow D$ satisfying for $x, y, z \in D$,
\begin{align*}
(x \prec y) \prec z =~& x \prec ( y \prec z + y \succ z + y \curlyvee z),\\
(x \succ y) \prec z =~& x \succ (y \prec z),\\
(x \prec y + x \succ y + x \curlyvee y) \succ z =~& x \succ (y \succ z),\\
( x \prec y + x \succ y + x \curlyvee y ) \curlyvee z  ~+~ (x \curlyvee y) \prec z =~& x \succ (y \curlyvee z) ~+~ x \curlyvee ( y \prec z + y \succ z + y \curlyvee z ).
\end{align*}
\end{defn}

Let $(D, \prec, \succ, \curlyvee)$ be an NS-algebra. Then $(D, \ast)$ is an associative algebra, where $$x \ast y = x \prec y + x \succ y + x \curlyvee y, \text{ for } x, y \in D.$$

Any dendriform algebra $(D, \prec, \succ)$ is obviously an NS-algebra with $\curlyvee = 0$.
The family version of NS-algebra is given by the following.

\begin{defn}\label{defn-ns-fam}
An {\bf NS-family algebra} is a vector space $D$ together with a family of bilinear operations $\{ \prec_\alpha, \succ_\alpha, \curlyvee_{\alpha, \beta} : D \otimes D \rightarrow D \}_{\alpha, \beta \in \Omega}$ satisfying
\begin{align}
(x \prec_\alpha y) \prec_\beta z =~& x \prec_{\alpha \beta} ( y \prec_\beta z + y \succ_\alpha z + y \curlyvee_{\alpha, \beta} z), \label{ns-fam1}\\
(x \succ_\alpha y) \prec_\beta z =~& x \succ_\alpha (y \prec_\beta z), \label{ns-fam2} \\
(x \prec_\beta y + x \succ_\alpha y + x \curlyvee_{\alpha, \beta} y) \succ_{\alpha \beta} z =~& x \succ_\alpha (y \succ_\beta z), \label{ns-fam3} \\
( x \prec_\beta y + x \succ_\alpha y + x \curlyvee_{\alpha, \beta} y ) \curlyvee_{\alpha \beta, \gamma} z  ~+~& (x \curlyvee_{\alpha, \beta} y) \prec_\gamma z \label{ns-fam4} \\
=~& x \succ_\alpha (y \curlyvee_{\beta, \gamma} z) ~+~ x \curlyvee_{\alpha, \beta \gamma} ( y \prec_\gamma z + y \succ_\beta z + y \curlyvee_{\beta, \gamma} z ),  \nonumber
\end{align}
for $x, y, z \in D$ and $\alpha, \beta, \gamma \in \Omega$.
\end{defn}

Let $(D, \{ \prec_\alpha, \succ_\alpha, \curlyvee_{\alpha, \beta} \}_{\alpha, \beta \in \Omega})$ and $(D', \{ \prec'_\alpha, \succ'_\alpha, \curlyvee'_{\alpha, \beta} \}_{\alpha, \beta \in \Omega})$ be two NS-family algebras. A {\bf morphism} between them is given by a linear map $f : D \rightarrow D'$ satisfying
\begin{align*}
f (x& \prec_\alpha y) = f(x) \prec'_\alpha f(y), ~~~~~ f (x \succ_\alpha y) = f(x) \succ'_\alpha f(y) ~~~ \text{ and } \\
&f (x \curlyvee_{\alpha, \beta} y) = f(x) \curlyvee'_{\alpha, \beta} f(y), ~ \text{ for } x, y \in D ~ \text{ and } \alpha, \beta \in \Omega.
\end{align*}

We denote the category of NS-family algebras and morphisms between them by {\bf NSf}. The next result finds the link between NS-family algebras and ordinary NS-algebras. This generalizes Proposition \ref{see} in the context of NS-family algebras.

\begin{thm}\label{ns-fam-ns}
Let $(D, \{ \prec_\alpha, \succ_\alpha, \curlyvee_{\alpha, \beta} \}_{\alpha, \beta \in \Omega})$ be an NS-family algebra. Then $(D \otimes {\bf k} \Omega, \prec, \succ, \curlyvee)$ is an NS-algebra, where
\begin{align*}
(x \otimes& \alpha)   \prec (y \otimes \beta) := (x \prec_\beta y ) \otimes \alpha \beta, \qquad
(x \otimes \alpha)   \succ (y \otimes \beta) := (x \succ_\alpha y) \otimes \alpha \beta, \\
&(x \otimes \alpha)   \curlyvee (y \otimes \beta) := (x \curlyvee_{\alpha, \beta} y) \otimes \alpha \beta, ~\text{ for } x \otimes \alpha,~ y \otimes \beta \in D \otimes {\bf k} \Omega.
\end{align*}

Further, if $f : D \rightarrow D'$ is a morphism between NS-family algebras from $(D, \{ \prec_\alpha, \succ_\alpha, \curlyvee_{\alpha, \beta} \}_{\alpha, \beta \in \Omega})$ to $(D', \{ \prec'_\alpha, \succ'_\alpha, \curlyvee'_{\alpha, \beta} \}_{\alpha, \beta \in \Omega})$, then the map $F : D \otimes {\bf k} \Omega \rightarrow D' \otimes {\bf k} \Omega$, $F (x \otimes \alpha) = f(x) \otimes \alpha$ is a morphism between induced NS-algebras.
\end{thm}

\begin{proof}
For $x \otimes \alpha$, $y \otimes \beta$, $z \otimes \gamma \in D \otimes {\bf k} \Omega$, we have
\begin{align*}
&\big( (x \otimes \alpha) \prec ( y \otimes \beta)  \big) \prec (z \otimes \gamma) \\
&= \big(   (x \prec_\beta y) \prec_\gamma z   \big) \otimes \alpha \beta \gamma \\
&\stackrel{(\ref{ns-fam1})}{=} \big(  x \prec_{\beta \gamma} (y \prec_\gamma z + y \succ_\beta z + y \curlyvee_{\beta, \gamma} z)  \big)   \otimes \alpha \beta \gamma \\
&= (x \otimes \alpha) \prec \big( ( y \prec_\gamma z + y \succ_\beta z + y \curlyvee_{\beta, \gamma} z ) \otimes \beta \gamma    \big) \\
&= (x \otimes \alpha) \prec \big( (y \otimes \beta) \prec (z \otimes \gamma) ~+~  (y \otimes \beta) \succ (z \otimes \gamma) ~+~ (y \otimes \beta) \curlyvee (z \otimes \gamma) \big).
\end{align*}
Similarly,
\begin{align*}
&\big(  (x \otimes \alpha) \succ (y \otimes \beta)  \big) \prec (z \otimes \gamma) \\
&= \big(  (x \succ_\alpha y) \prec_\gamma z  \big) \otimes \alpha \beta \gamma \\
&\stackrel{(\ref{ns-fam2})}{=} \big( x \succ_\alpha (y \prec_\gamma z)  \big) \otimes \alpha \beta \gamma \\
&= (x \otimes \alpha) \succ \big( (y \otimes \beta) \prec (z \otimes \gamma)   \big),
\end{align*}
and
\begin{align*}
&\big(  (x \otimes \alpha) \prec (y \otimes \beta) + (x \otimes \alpha) \succ (y \otimes \beta) + (x \otimes \alpha) \curlyvee (y \otimes \beta) \big) \succ (z \otimes \gamma) \\
&= \big(  ( x \prec_\beta y + x \succ_\alpha y + x \curlyvee_{\alpha, \beta} y  ) \otimes \alpha \beta    \big) \succ (z \otimes \gamma) \\
&= \big(  ( x \prec_\beta y + x \succ_\alpha y + x \curlyvee_{\alpha, \beta} y  ) \succ_{\alpha \beta} z  \big) \otimes \alpha \beta \gamma \\
&\stackrel{(\ref{ns-fam3})}{=} \big(  x \succ_\alpha (y \succ_\beta z)  \big) \otimes \alpha \beta \gamma \\ 
&= (x \otimes \alpha) \succ \big( (y \otimes \beta) \succ (z \otimes \gamma)  \big).
\end{align*}
Finally,
\begin{align*}
&\big( (x \otimes \alpha) \prec (y \otimes \beta)  ~+~ (x \otimes \alpha) \succ (y \otimes \beta) ~+~ (x \otimes \alpha) \curlyvee (y \otimes \beta)  \big) \curlyvee (z \otimes \gamma) + \big(   (x \otimes \alpha) \curlyvee (y \otimes \beta)  \big) \prec (z \otimes \gamma) \\
&= \big(  (  x \prec_\beta y + x \succ_\alpha y + x \curlyvee_{\alpha, \beta} y) \otimes \alpha \beta  \big) \curlyvee (z \otimes \gamma) ~+~ \big(  (x \curlyvee_{\alpha, \beta} y) \otimes \alpha \beta \big) \prec (z \otimes \gamma) \\
&= \big( (  x \prec_\beta y + x \succ_\alpha y + x \curlyvee_{\alpha, \beta} y) \curlyvee_{\alpha \beta, \gamma} z  ~+~ (x \curlyvee_{\alpha, \beta} y) \prec_\gamma z \big) \otimes \alpha \beta \gamma \\
&\stackrel{(\ref{ns-fam4})}{=} \big( x \succ_\alpha (y \curlyvee_{\beta, \gamma} z) ~+~ x \curlyvee_{\alpha, \beta \gamma} (y \prec_\gamma z + y \succ_\beta z + y \curlyvee_{\beta, \gamma} z)   \big) \otimes \alpha \beta \gamma \\
&= (x \otimes \alpha) \succ \big( (y \curlyvee_{\beta, \gamma} z) \otimes \beta \gamma   \big) ~+~ (x \otimes \alpha) \curlyvee \big( (y \prec_\gamma z + y \succ_\beta z + y \curlyvee_{\beta, \gamma} z) \otimes \beta \gamma   \big) \\
&= (x \otimes \alpha) \succ \big( (y \otimes \beta) \curlyvee (z \otimes \gamma) \big) + (x \otimes \alpha) \curlyvee \big(   (y \otimes \beta) \prec (z \otimes \gamma) +   (y \otimes \beta) \succ (z \otimes \gamma)  +  (y \otimes \beta) \curlyvee (z \otimes \gamma)  \big).
\end{align*}
This proves that $(D \otimes {\bf k} \Omega, \prec, \succ, \curlyvee)$ is an NS-algebra. This completes the first part.

The second part is straightforward.
\end{proof}

\begin{exam}
Any dendriform family algebra $(D, \{ \prec_\alpha, \succ_\alpha \}_{\alpha \in \Omega})$ can be regarded as an NS-family algebra $(D, \{ \prec_\alpha, \succ_\alpha , \curlyvee_{\alpha, \beta} \}_{\alpha, \beta \in \Omega})$ with $\curlyvee_{\alpha, \beta} = 0$ for all $\alpha, \beta \in \Omega$.
\end{exam}

In \cite{zhang-free} the authors introduced tridendriform family algebras (a generalization of dendriform family algebras) and show that such algebras arise from Rota-Baxter families of arbitrary weight. In the following, we observe that a tridendriform family algebra can also be regarded as an NS-family algebra.

\begin{defn}
A {\bf tridendriform family algebra} is a vector space $D$ together with a family of bilinear operations $\{ \prec_\alpha, \succ_\alpha : D \otimes D \rightarrow D \}_{\alpha \in \Omega}$ and a bilinear map ~$ \odot : D \otimes D \rightarrow D$ satisfying
\begin{align}
( x \prec_\alpha y) \prec_\beta z =~& x \prec_{\alpha \beta} (y \prec_\beta z + y \succ_\alpha z + y \odot z),\\
(x \succ_\alpha y) \prec_\beta z =~& x \succ_\alpha (y \prec_\beta z),\\
(x \prec_\beta y + x \succ_\alpha y + x \odot y) \succ_{\alpha \beta} z =~& x \succ_\alpha (y \succ_\beta z),\\
(x \succ_\alpha y) \odot z =~& x \succ_\alpha (y \odot z), \label{tf4} \\
(x \prec_\alpha y) \odot z =~& x \odot (y \succ_\alpha z), \label{tf5} \\
(x \odot y) \prec_\alpha z =~& x \odot (y \prec_\alpha z), \label{tf6} \\
(x \odot y) \odot z =~& x \odot (y \odot z),~ \text{ for } x, y, z \in D \text{ and } \alpha, \beta \in \Omega. \label{tf7}
\end{align}
\end{defn}

Let $A$ be an associative algebra. A collection $\{ R_\alpha : A \rightarrow A \}_{\alpha \in \Omega}$ of linear maps is said to be a Rota-Baxter family of weight $\lambda$ if they satisfy
\begin{align*}
R_\alpha (a) \cdot R_\beta (b) = R_{\alpha \beta} \big( R_\alpha (a) \cdot b + a \cdot R_\beta (b) + \lambda a \cdot b ),~\text{for } a, b \in A \text{ and } \alpha, \beta \in \Omega.
\end{align*}
A Rota-Baxter family of weight $\lambda$ induces a tridendriform family algebra given by
\begin{align*}
a \prec_\alpha b = a \cdot R_\alpha (a), \quad a \succ_\alpha b = R_\alpha (a) \cdot b ~~ \text{ and } ~~ a \odot b = \lambda a \cdot b, \text{ for } a, b \in A \text{ and } \alpha \in \Omega.
\end{align*}

\begin{prop} Let $(D, \{ \prec_\alpha, \succ_\alpha \}_{\alpha \in \Omega}, \odot)$ be a tridendriform family algebra. Then it can be considered as an NS-family algebra $( D, \{ \prec_\alpha, \succ_\alpha, \curlyvee_{\alpha, \beta} \}_{\alpha, \beta \in \Omega})$, where $x \curlyvee_{\alpha, \beta} y = x \odot y$, for all $x, y \in D$ and $\alpha, \beta \in \Omega$.

Let $A$ be an associative algebra and $\{ R_\alpha : A \rightarrow A \}_{\alpha \in \Omega}$ be a Rota-Baxter family of weight $\lambda$. Then $(A, \{ \prec_\alpha, \succ_\alpha, \curlyvee_{\alpha, \beta} \}_{\alpha, \beta \in \Omega})$ is an NS-family algebra, where
\begin{align*}
a \prec_\alpha b = a \cdot R_\alpha (b), ~~~ a \succ_\alpha b = R_\alpha (a) \cdot b ~ \text{ and } ~ a \curlyvee_{\alpha, \beta} b = \lambda a \cdot b, \text{ for } a, b \in A \text{ and } \alpha, \beta \in \Omega.
\end{align*}
\end{prop}

\begin{proof}
Note that the first three identities (\ref{ns-fam1})-(\ref{ns-fam3}) of the NS-family algebra follows from the first three identities of tridendriform family algebra. Finally, we have
\begin{align*}
&( x \prec_\beta y + x \succ_\alpha y + x \curlyvee_{\alpha, \beta} y) \curlyvee_{\alpha \beta, \gamma} z  ~+~ (x \curlyvee_{\alpha, \beta} y) \prec_\gamma z \\
&= ( x \prec_\beta y + x \succ_\alpha y + x \odot y) \odot z  ~+~ (x \odot y) \prec_\gamma z \\
&= x \odot (y \succ_\beta z) + x \succ_\alpha (y \odot z) + x \odot (y \odot z) + x \odot (y \prec_\gamma z) ~~~~ (\text{by } (\ref{tf4})-(\ref{tf7})) \\
&= x \curlyvee_{\alpha, \beta \gamma} (y \succ_\beta z) + x \succ_\alpha (y \curlyvee_{\beta, \gamma} z) + x \curlyvee_{\alpha, \beta \gamma} (y \curlyvee_{\beta, \gamma} z) + x \curlyvee_{\alpha, \beta \gamma} (y \prec_\gamma z) \\
&= x \succ_\alpha (y \curlyvee_{\beta, \gamma} z) ~+~ x \curlyvee_{\alpha, \beta \gamma} (y \prec_\gamma z + y \succ_\beta z + y \curlyvee_{\beta, \gamma} z).
\end{align*}
This proves that $(D, \{ \prec_\alpha, \succ_\alpha, \curlyvee{\alpha, \beta} \}_{\alpha, \beta \in \Omega})$ is an NS-family algebra. The second part follows from the first one.
\end{proof}

\medskip

In \cite{leroux} the author showed that a Nijenhuis operator on an associative algebra induces an NS-algebra structure. In the next result, we show that a Nijenhuis family induces an NS-family algebra.

\begin{prop}
Let $A$ be an associative algebra and $\{ N_\alpha : A \rightarrow A \}_{\alpha \in \Omega}$ be a Nijenhuis family. Then $(A, \{ \prec_\alpha, \succ_\alpha, \curlyvee_{\alpha, \beta} \}_{\alpha, \beta \in \Omega})$ is an NS-family algebra, where
\begin{align*}
a \prec_\alpha b := a \cdot N_\alpha (b), ~~~~ a \succ_\alpha b := N_\alpha (a) \cdot b ~~~~ \text{ and } ~~~~ a \curlyvee_{\alpha, \beta} b := - N_{\alpha \beta} (a \cdot b), ~ \text{ for } a, b \in A.
\end{align*}
\end{prop}

\begin{proof}
The verifications of the identities (\ref{ns-fam1})-(\ref{ns-fam3}) are straightforward using the definition of Nijenhuis family. Here we only verify the identity (\ref{ns-fam4}). For any $a, b, c \in A$ and $\alpha, \beta, \gamma \in \Omega$,
\begin{align*}
&(a \prec_\beta b + a \succ_\alpha b + a \curlyvee_{\alpha, \beta} b) \curlyvee_{\alpha \beta, \gamma} c + (a \curlyvee_{\alpha, \beta} b) \prec_\gamma c \\
&= - N_{\alpha \beta \gamma} \big( (  N_\alpha (a) \cdot b + a \cdot N_\beta (b) - N_{\alpha \beta} (a \cdot b)) \cdot c   \big) - N_{\alpha \beta} (a \cdot b) \cdot N_\gamma (c) \\
&= - N_{\alpha \beta \gamma} \big( N_\alpha (a) \cdot b \cdot c + a \cdot N_\beta (b) \cdot c - \cancel{N_{\alpha \beta} (a \cdot b) \cdot c}   \big) - N_{\alpha \beta \gamma} \big(   \cancel{N_{\alpha \beta} (a \cdot b) \cdot c} + a \cdot b \cdot N_\gamma (c) - N_{\alpha \beta \gamma} (a \cdot b \cdot c) \big) \\
&= - N_{\alpha \beta \gamma} \big( N_\alpha (a) \cdot b \cdot c + {a \cdot N_{\beta \gamma} (b \cdot c)} - N_{\alpha \beta \gamma} (a \cdot b \cdot c)  \big) - N_{\alpha \beta \gamma} \big(  a \cdot N_\beta (b) \cdot c + a \cdot b \cdot N_\gamma (c) - {a \cdot N_{\beta \gamma} (b \cdot c)} \big) \\
&= - N_\alpha (a) \cdot N_{\beta \gamma} (b \cdot c) - N_{\alpha \beta \gamma} \big(   a \cdot (N_\beta (b) \cdot c + b \cdot N_\gamma (c) - N_{\beta \gamma} (b \cdot c)) \big) \\
&= a \succ_\alpha (b \curlyvee_{\beta, \gamma} c) + a \curlyvee_{\alpha, \beta \gamma} (b \prec_\gamma c + b \succ_\beta c + b \curlyvee_{\beta, \gamma} c).
\end{align*}
This proves that $(A, \{ \prec_\alpha, \succ_\alpha, \curlyvee_{\alpha, \beta} \}_{\alpha, \beta \in \Omega})$ is an NS-family algebra.
\end{proof}

In the following, we show that an $H$-twisted $\mathcal{O}$-operator family induces an NS-family algebra. In particular, an $\mathcal{O}$-operator family induces a dendriform family algebra (the case $H=0$). Hence we recover Proposition \ref{prop-o-dend}.

\begin{prop}\label{twisted-rota-ns}
Let $A$ be an associative algebra, $M$ be an $A$-bimodule and $H$ be a Hochschild $2$-cocycle. Let $\{ T_\alpha : M \rightarrow A \}_{\alpha \in \Omega}$ be an $H$-twisted $\mathcal{O}$-operator family. Then $(M, \{ \prec_\alpha, \succ_\alpha, \curlyvee_{\alpha, \beta} \}_{\alpha, \beta \in \Omega})$ is an NS-family algebra, where
\begin{align*}
u \prec_\alpha v := u \cdot T_\alpha (v), ~~~ u \succ_\alpha v := T_\alpha (u) \cdot v ~~~ \text{ and } ~~~ u \curlyvee_{\alpha, \beta} v := H (T_\alpha (u), T_\beta (v)), ~\text{ for } u, v \in M.
\end{align*}
Further, if $\{ T_\alpha : M \rightarrow A \}_{\alpha \in \Omega}$ and $\{ T'_\alpha : M' \rightarrow A' \}_{\alpha \in \Omega}$ are two twisted $\mathcal{O}$-operator families and $(\phi, \psi)$ is a morphism between them (cf. Definition \ref{tw-o-map}), then the map $\psi$ is a {morphism} between the induced NS-family algebras. Therefore, there is a functor $\mathcal{P} : {\bf TwOoperf} \rightarrow {\bf NSf}$.
\end{prop}

\begin{proof}
For any $u, v, w \in M$, we have
\begin{align*}
(u \prec_\alpha v) \prec_\beta w =~& u \cdot T_\alpha (v) \cdot T_\beta (w) \\
\stackrel{(\ref{twisted-o-fam})}{=}~& u \cdot T_{\alpha \beta} \big(   T_\alpha (v) \cdot w + v \cdot T_\beta (w) + H (T_\alpha (v), T_\beta (w)) \big) \\
=~& u \prec_{\alpha \beta} (v \prec_\beta w + v \succ_\alpha w + v \curlyvee_{\alpha, \beta} w).
\end{align*}
Similarly, we have
\begin{align*}
(u \succ_\alpha v) \prec_\beta w = (T_\alpha (u) \cdot v) \cdot T_\beta (w) = T_\alpha (u) \cdot (v \cdot T_\beta (w)) = u \succ_\alpha (v \prec_\beta w),
\end{align*}
\begin{align*}
(u \prec_\beta v + u \succ_\alpha v + u \curlyvee_{\alpha, \beta} v) \succ_{\alpha \beta} w =~& T_{\alpha \beta} \big( u \cdot T_\beta (v) + T_\alpha (u) \cdot v + H (T_\alpha (u), T_\beta (v) ) \big) \cdot w \\
\stackrel{(\ref{twisted-o-fam})}{=}& T_\alpha (u) \cdot T_\beta (v) \cdot w \\
=~& u \succ_\alpha (v \succ_\beta w).
\end{align*}
Thus, we verify the three conditions (\ref{ns-fam1})-(\ref{ns-fam3}). To prove the condition (\ref{ns-fam4}), we observe that
\begin{align}\label{h-abg}
T_\alpha (u) \cdot H (T_\beta (v), T_\gamma (w)) &- H (T_\alpha (u) \cdot T_\beta (v), T_\gamma (w)) \\&+ H (T_\alpha (u) , T_\beta (v) \cdot T_\gamma (w)) - H (T_\alpha (u), T_\beta (v)) \cdot T_\gamma (w) = 0 \nonumber
\end{align}
as $H$ is a $2$-cocycle. Note that (\ref{h-abg}) is equivalent to
\begin{align*}
u \succ_\alpha (v \curlyvee_{\beta, \gamma} w) -~& (u \prec_\beta v + u \succ_\alpha v + u \curlyvee_{\alpha, \beta} v) \curlyvee_{\alpha \beta, \gamma} w \\
+~& u \curlyvee_{\alpha, \beta \gamma} (v \prec_\gamma w + v \succ_\beta w + v \curlyvee_{\beta, \gamma} w) - (u \curlyvee_{\alpha, \beta} v) \prec_\gamma w = 0.
\end{align*}
This shows that the condition (\ref{ns-fam4}) also holds. Hence $(M, \{ \prec_\alpha, \succ_\alpha, \curlyvee_{\alpha, \beta} \}_{\alpha, \beta \in \Omega})$ is an NS-family algebra.

For the second part, we see
\begin{align*}
&\psi (u \prec_\alpha v) = \psi (u \cdot T_\alpha (u)) = \psi (u) \cdot \phi T_\alpha (v) = \psi (u) \cdot T'_\alpha (\psi (v)) = \psi (u) \prec'_\alpha \psi (v), \\
&\psi (u \succ_\alpha v) = \psi (T_\alpha (u) \cdot v) = \phi T_\alpha (u) \cdot \psi (v) = T'_\alpha (\psi (u)) \cdot \psi (v) = \psi (u) \succ'_\alpha \psi (v), \\
&\psi (u \curlyvee_{\alpha, \beta} v) = \psi \big(  H (T_\alpha (u), T_\beta (v) )  \big) = H' (\phi T_\alpha (u), \phi T_\beta (v) ) = H' \big( T'_\alpha (\psi (u)),  T'_\beta (\psi (v)) \big) = \psi (u ) \curlyvee'_{\alpha, \beta} \psi (v). 
\end{align*}
This proves that $\psi$ is a morphism between induced NS-family algebras.
\end{proof}


Let $(D, \{ \prec_\alpha, \succ_\alpha, \curlyvee_{\alpha, \beta} \}_{\alpha, \beta \in \Omega} )$ be an NS-family algebra. Then by Proposition \ref{ns-fam-ns}, $(D \otimes {\bf k} \Omega, \prec, \succ, \curlyvee)$ is an ordinary NS-algebra. Hence $(D \otimes {\bf k} \Omega, \ast)$ is an associative algebra, where
\begin{align*}
(a \otimes \alpha) \ast (b \otimes \beta) := (a \succ_\alpha b + a \prec_\beta b + a \curlyvee_{\alpha, \beta} b) \otimes \alpha \beta.
\end{align*}
We denote this algebra by $(D \otimes {\bf k} \Omega)_\mathrm{Tot}$. Moreover, there is a $(D \otimes {\bf k} \Omega)_\mathrm{Tot}$-bimodule structure on $D$ with the left and right actions
\begin{align*}
(a \otimes \alpha) \cdot b = a \succ_\alpha b ~~~ \text{ and } ~~~ b \cdot (a \otimes \alpha) = b \prec_\alpha a,
\end{align*}
for $a \otimes \alpha \in D \otimes {\bf k} \Omega$ and $b \in D.$ We also consider a map $H: (D \otimes {\bf k} \Omega)_\mathrm{Tot} \otimes (D \otimes {\bf k} \Omega)_\mathrm{Tot} \rightarrow D$ given by $H (a \otimes \alpha, b \otimes \beta) = a \curlyvee_{\alpha, \beta} b$. Then the condition (\ref{ns-fam4}) is equivalent to the fact that the map $H$ is a Hochschild $2$-cocycle. With all these notations, the collection $\{ \mathrm{Id}_\alpha : D \rightarrow (D \otimes {\bf k} \Omega)_\mathrm{Tot} \}_{\alpha \in \Omega}$, where $\mathrm{Id}_\alpha (x) = x \otimes \alpha$, is an $H$-twisted $\mathcal{O}$-operator family. Moreover, the induced NS-family algebra structure on $D$ coincides with the given one. The above correspondence is functorial. Hence, we get a functor $\mathcal{Q} : {\bf NSf} \rightarrow {\bf TwOoperf}.$

\medskip

The proof of the following result is similar to Proposition \ref{prop-functor}. Hence we will not repeat it here.

\begin{prop}
The functor $\mathcal{Q} : {\bf NSf} \rightarrow {\bf TwOoperf}$ constructed above is left adjoint to the functor $\mathcal{P} : {\bf TwOoperf} \rightarrow {\bf NSf}$.
\end{prop}

\medskip

Let $\{ T_\alpha : M \rightarrow A \}_{\alpha \in \Omega}$ be an $H$-twisted $\mathcal{O}$-operator family. Then by Proposition \ref{twisted-rota-ns}, we get that $(M, \{ \prec_\alpha, \succ_\alpha, \curlyvee_{\alpha, \beta} \}_{\alpha, \beta \in \Omega})$ is an NS-family algebra. Hence, by Theorem \ref{ns-fam-ns}, $(M \otimes {\bf k} \Omega, \prec, \succ, \curlyvee)$ is an NS-algebra, where
\begin{align}\label{diag-ns}
\begin{cases}
(u \otimes \alpha) \prec (u \otimes \beta) = (u \prec_\beta v) \otimes \alpha \beta = (u \cdot T_\beta (v)) \otimes \alpha \beta,\\
(u \otimes \alpha) \succ (u \otimes \beta) = (u \succ_\alpha v) \otimes \alpha \beta = (T_\alpha (u) \cdot v) \otimes \alpha \beta,\\
(u \otimes \alpha) \curlyvee (u \otimes \beta) = (u \curlyvee_{\alpha, \beta} v) \otimes \alpha \beta = H (T_\alpha (u), T_\beta (v)) \otimes \alpha \beta.
\end{cases}
\end{align}
On the other hand, the $H$-twisted $\mathcal{O}$-operator family $\{ T_\alpha : M \rightarrow A \}_{\alpha \in \Omega}$ induces the $\widehat{H}$-twisted $\mathcal{O}$-operator $T : M \otimes {\bf k} \Omega \rightarrow A \otimes {\bf k} \Omega$, $T (u \otimes \alpha) = T_\alpha (u) \otimes \alpha$, for $u \otimes \alpha \in M \otimes {\bf k} \Omega$ (see Proposition \ref{of-o}). Hence, we get an NS-algebra structure on $M \otimes {\bf k} \Omega$, which coincides with the one given by (\ref{diag-ns}). As a summary, we obtain the commutative diagram (\ref{comm-diag}) given in Section \ref{subsec-o}.

\section{Cohomology and deformations of family structures}\label{sec-5}
In this section, we introduce the cohomology of twisted $\mathcal{O}$-operator families and NS-family algebras. We observe that these cohomologies govern the corresponding formal deformations.

\subsection{Associative algebras relative to a semigroup}
In this subsection, we first consider the notion of associative algebra relative to a semigroup $\Omega$ (which we call an $\Omega$-associative algebra) introduced by Aguiar \cite{agu}. Such algebras are closely related with (twisted) $\mathcal{O}$-operator family and NS-family algebras. We define bimodules over such algebras and introduce cohomology theory. We will use this to define the cohomology of twisted $\mathcal{O}$-operator families in the next subsection.

\begin{defn} \cite{agu}
An associative algebra relative to a semigroup $\Omega$ ({\bf $\Omega$-associative algebra}) is a vector space $A$ together with a collection of bilinear operations $\{  ~ \cdot_{\alpha, \beta} : A \otimes  A \rightarrow  A \}_{\alpha, \beta \in \Omega}$ satisfying
\begin{align*}
(a \cdot_{\alpha , \beta} b) \cdot_{\alpha \beta, \gamma} c = a \cdot_{\alpha, \beta \gamma} (b \cdot_{\beta, \gamma} c), ~ \text{ for } a, b, c \in A \text{ and } \alpha, \beta, \gamma \in \Omega.
\end{align*}
\end{defn}

Aguiar showed that a dendriform family algebra induces an $\Omega$-associative algebra. We show that an NS-family algebra (cf. Definition \ref{defn-ns-fam}) induces an $\Omega$-associative algebra which generalizes the result of Aguiar.

\begin{prop}\label{ns-fam-omega}
Let $(D, \{ \prec_\alpha, \succ_\alpha, \curlyvee_{\alpha, \beta} \}_{\alpha, \beta \in \Omega})$ be an NS-family algebra. Then $(D, \{ ~\ast_{\alpha, \beta} \}_{\alpha, \beta \in \Omega})$ is an $\Omega$-associative algebra, where
\begin{align*}
x \ast_{\alpha, \beta} y = x \prec_\beta y + x \succ_\alpha y + x \curlyvee_{\alpha, \beta} y,~ \text{ for } x, y \in D \text{ and } \alpha, \beta \in \Omega.
\end{align*}
\end{prop}

\begin{proof}
In the identity (\ref{ns-fam1}), replacing $\alpha$ by $\beta$ and $\beta$ by $\gamma$, we get
\begin{align}\label{ns-fam-1n}
(x \prec_\beta y) \prec_\gamma z = x \prec_{\beta \gamma} (y \prec_\gamma z + y \succ_\beta + y \curlyvee_{\beta, \gamma} z).
\end{align}
Similarly, in the identity (\ref{ns-fam2}), replacing $\beta$ by $\gamma$, we get
\begin{align}\label{ns-fam-2n}
(x \succ_\alpha y) \prec_\gamma z = x \succ_\alpha (y \prec_\gamma z).
\end{align}
Finally, if we add the left hand sides of the identities (\ref{ns-fam-1n}), (\ref{ns-fam-2n}), (\ref{ns-fam3}) and (\ref{ns-fam4}), we get $(x \ast_{\alpha, \beta} y)\ast_{\alpha \beta, \gamma} z$. On the other hand, if we add the right hand sides of the above mentioned identities, we get $x \ast_{\alpha, \beta \gamma} (y \ast_{\beta, \gamma} z)$. This shows that $(D, \{ ~\ast_{\alpha, \beta} \}_{\alpha, \beta \in \Omega})$ is an $\Omega$-associative algebra.
\end{proof}

Combining Proposition \ref{twisted-rota-ns} and Proposition \ref{ns-fam-omega}, we get the following.

\begin{prop}\label{omega-tw}
Let $\{ T_\alpha : M \rightarrow A \}_{\alpha \in \Omega}$ be an $H$-twisted $\mathcal{O}$-operator family. Then $(M , \{~ \ast_{\alpha, \beta} \}_{\alpha, \beta \in \Omega})$ is an $\Omega$-associative algebra, where
\begin{align}\label{omega-tw-for}
u \ast_{\alpha, \beta} v = T_\alpha (u) \cdot v + u \cdot T_\beta (v) + H (T_\alpha (u), T_\beta (v)), ~ \text{ for } u, v \in M \text{ and } \alpha, \beta \in \Omega.
\end{align}
\end{prop}

Let $(A , \{ ~ \cdot_{\alpha, \beta} \}_{\alpha, \beta \in \Omega})$ be an $\Omega$-associative algebra. A {\bf bimodule} over it consists of a vector space $M$ together with a collection 

\[ \begin{Bmatrix} l_{\alpha, \beta} : A \otimes M \rightarrow M ,~ (a, u) \mapsto a \cdot_{\alpha, \beta } u \\
r_{\alpha, \beta} : M \otimes A \rightarrow M, ~ (u, a) \mapsto u \cdot_{\alpha, \beta} a
\end{Bmatrix}_{\alpha, \beta \in \Omega}\]
of bilinear maps satisfying for $a, b \in A, u \in M \text{ and } \alpha, \beta, \gamma \in \Omega$,
\begin{align*}
(a \cdot_{\alpha , \beta} b) \cdot_{\alpha \beta, \gamma} u = a \cdot_{\alpha, \beta \gamma} (b \cdot_{\beta, \gamma} u), ~~
(a \cdot_{\alpha , \beta} u) \cdot_{\alpha \beta, \gamma} b = a \cdot_{\alpha, \beta \gamma} (u \cdot_{\beta, \gamma} b), ~~ (u \cdot_{\alpha , \beta} a) \cdot_{\alpha \beta, \gamma} b = u \cdot_{\alpha, \beta \gamma} (a \cdot_{\beta, \gamma} b).
\end{align*}

\noindent \underline{{\bf Assumpsion}}. For the rest of the present subsection and subsection \ref{subsec-tw-co}, we assume that $\Omega$ is a semigroup with unit $1 \in \Omega$. The unital condition of $\Omega$ is only useful in the coboundary operator of the cohomology at the degree $0$ level.

Let $(A , \{ ~ \cdot_{\alpha, \beta} \}_{\alpha, \beta \in \Omega})$ be an $\Omega$-associative algebra and $M$ be a bimodule over it. For each $n \geq 0$, we define an abelian group $C^n_{\Omega\mathrm{-Hoch}} (A,M)$ by
\begin{align*}
&C^0_{\Omega\mathrm{-Hoch}} (A,M) = M,\\
&C^{n \geq 1}_{\Omega\mathrm{-Hoch}} (A,M) = \big\{  f = \{ f_{\alpha_1, \ldots, \alpha_n} \}_{\alpha_1, \ldots, \alpha_n \in \Omega} ~|~ f_{\alpha_1, \ldots, \alpha_n} : A^{\otimes n} \rightarrow M \text{ is multilinear}  \big\}.
\end{align*}
There is a  map $\delta_{\Omega\mathrm{-Hoch}} : C^n_{\Omega\mathrm{-Hoch}} (A,M) \rightarrow C^{n+1}_{\Omega\mathrm{-Hoch}} (A,M)$ given by
\begin{align}
\big(   \delta_{\Omega\mathrm{-Hoch}} (u) \big)_\alpha (a) =~& a \cdot_{\alpha, 1} u ~-~ u \cdot_{1, \alpha} a, \\
\big(   \delta_{\Omega\mathrm{-Hoch}} (f) \big)_{\alpha, \ldots, \alpha_{n+1}} (a_1, \ldots, a_{n+1} ) =~& a_1 \cdot_{\alpha_1 , \alpha_2 \cdots \alpha_{n+1}} f_{\alpha_2, \ldots, \alpha_{n+1}} (a_2, \ldots, a_{n+1})  \\
+ \sum_{i=1}^n (-1)^i ~ f_{\alpha_1, \ldots, \alpha_i \alpha_{i+1}, \ldots, \alpha_{n+1}} & \big(  a_1, \ldots, a_{i-1}, a_i \cdot_{\alpha_i \alpha_{i+1}} a_{i+1}, a_{i+2}, \ldots, a_{n+1}  \big) \nonumber \\
+ (-1)^{n+1} ~ f_{\alpha_1, \ldots, \alpha_n} &(a_1, \ldots, a_n) ~\cdot_{\alpha_1 \cdots a_n, \alpha_{n+1}} a_{n+1}, \nonumber
\end{align}
for $m \in M = C^0_{\Omega\mathrm{-Hoch}} (A,M)$ and $f \in C^{n \geq 1}_{\Omega\mathrm{-Hoch}} (A,M)$. Then similar to the standard Hochschild coboundary operator, one can show that $( \delta_{\Omega\mathrm{-Hoch}} )^2 = 0$. In other words, $ \{  C^{n \geq 0}_{\Omega\mathrm{-Hoch}} (A,M) , \delta_{\Omega\mathrm{-Hoch}} \}$ is a cochain complex. The corresponding cohomology groups are called the cohomology of the $\Omega$-associative algebra $A$ with coefficients in the bimodule $M$.

\subsection{Cohomology and deformations of twisted $\mathcal{O}$-operator family}\label{subsec-tw-co} Let $A$ be an associative algebra, $M$ be an $A$-bimodule and $H$ be a $2$-cocycle. Let $\{ T_\alpha : M \rightarrow A \}_{\alpha \in \Omega}$ be an $H$-twisted $\mathcal{O}$-operator family. Then we have seen in Proposition \ref{omega-tw} that $(M, \{ ~ \ast_{\alpha, \beta \}_{\alpha, \beta \in \Omega}})$ is an $\Omega$-associative algebra, where the structure maps $\{ ~ \ast_{\alpha, \beta} \}_{\alpha, \beta \in \Omega}$ are given by (\ref{omega-tw-for}). Note that these structure maps satisfy $T_{\alpha \beta} (u \ast_{\alpha, \beta} v) = T_\alpha (u) \cdot T_\beta (v)$, for $u, v \in M$ and $\alpha, \beta \in \Omega$. We define a collection of bilinear maps as follows

\[ \begin{Bmatrix} \triangleright_{\alpha, \beta} : M \otimes A \rightarrow A ,~ u \triangleright_{\alpha, \beta} a := T_\alpha (u) \cdot a - T_{\alpha \beta} (u \cdot a) - T_{\alpha \beta} \big( H (T_\alpha (u), a) \big) \\
\triangleleft_{\alpha, \beta} : A \otimes M \rightarrow A, ~ a \triangleleft_{\alpha, \beta} u := a \cdot T_\beta (u) - T_{\alpha \beta} (a \cdot u) - T_{\alpha \beta} \big(  H (a, T_\beta (u)) \big)
\end{Bmatrix}_{\alpha, \beta \in \Omega}.\]

\begin{thm}\label{bim-thm}
With the above notations, $(A, \{ \triangleright_{\alpha, \beta}, \triangleleft_{\alpha, \beta} \}_{\alpha, \beta \in \Omega})$ is a bimodule over the $\Omega$-associative algebra $(M, \{~ \cdot_{\alpha, \beta} \}_{\alpha, \beta \in \Omega})$.
\end{thm}

\begin{proof}
For any $u, v \in M$, $a \in A$ and $\alpha, \beta, \gamma \in \Omega$, we have
\begin{align*}
&( u \ast_{\alpha, \beta} v) \triangleright_{\alpha \beta, \gamma} a ~-~ u \triangleright_{\alpha, \beta \gamma} (v \triangleright_{\beta, \gamma} a) \\
&= T_{\alpha \beta} (u \ast_{\alpha, \beta} v ) \cdot a - T_{\alpha \beta \gamma} ( (u \ast_{\alpha, \beta} v) \cdot a) - T_{\alpha \beta \gamma} \big(  H ( T_{\alpha \beta} (u \ast_{\alpha, \beta} v), a) \big)  \\
& \qquad \qquad \qquad - u \triangleright_{\alpha, \beta \gamma} \big(    T_\beta (v) \cdot a - T_{\beta \gamma} (v \cdot a) - T_{\beta \gamma} (H (T_\beta v, a ) )\big) \\
& = \cancel{T_\alpha (u) \cdot T_\beta (v) \cdot a} - T_{\alpha \beta \gamma} (T_\alpha (u) \cdot v \cdot a) - \cancel{T_{\alpha \beta \gamma} (u \cdot T_\beta (v) \cdot a)} - T_{\alpha \beta \gamma} \big( H(T_\alpha u, T_\beta v) \cdot a  \big) \\
& \quad - T_{\alpha \beta \gamma} \big(  H (T_\alpha (u) \cdot T_\beta (v), a )  \big) - \cancel{T_\alpha (u) \cdot T_\beta (v) \cdot a} + T_\alpha (u) \cdot T_{\beta \gamma} (v \cdot a) + T_\alpha (u) \cdot T_{\beta \gamma} (H (T_\beta v, a)) \\
 & \quad + \cancel{T_{\alpha \beta \gamma} (u \cdot T_\beta (v) \cdot a)} - T_{\alpha \beta \gamma} (u \cdot T_{\beta \gamma} (v \cdot a)) - T_{\alpha \beta \gamma} \big( u \cdot T_{\beta \gamma} (H (T_\beta v, a )) \big) \\
& \quad + T_{\alpha \beta \gamma} \big( H ( T_\alpha (u), T_\beta (v) \cdot a) \big) - T_{\alpha \beta \gamma} \big( H (T_\alpha u, T_{\beta \gamma} (v \cdot a) )  \big) - T_{\alpha \beta \gamma} \big(  H (T_\alpha u, T_{\beta \gamma} H (T_\beta v, a) )  \big) \\
&= T_{\alpha \beta \gamma} \bigg(  - T_\alpha (u) \cdot v \cdot a - H (T_\alpha u, T_\beta v) \cdot a - H (T_\alpha (u) \cdot T_\beta (v), a) \big) + T_\alpha (u) \cdot v \cdot a + u \cdot T_{\beta \gamma}(v \cdot a) \\
& \quad + H (T_\alpha u, T_{\beta \gamma} (v \cdot a)) + T_\alpha (u) \cdot H (T_\beta v, a) + u \cdot T_{\beta \gamma} H (T_\beta v, a) + H (T_\alpha u, T_{\beta \gamma} H (T_\beta v, a))  - u \cdot T_{\beta \gamma} (v \cdot a) \\
& \quad - u \cdot T_{\beta \gamma} H (T_\beta v, a) + H (T_\alpha u, T_\beta (v) \cdot a) - H (T_\alpha u, T_{\beta \gamma} (v \cdot a)) - H \big(  T_\alpha u, T_{\beta \gamma} H (T_\beta v, a) \big) \bigg) \\
&= T_{\alpha \beta \gamma} \big(  - H (T_\alpha u, T_\beta v) \cdot a - H (T_\alpha (u) \cdot T_\beta (v), a) + T_\alpha (u) \cdot H (T_\beta v, a) + H (T_\alpha u, T_\beta (v) \cdot a ) \big) \\
&= T_{\alpha \beta \gamma} \big( (\delta_{{\mathrm{Hoch}}} H) (T_\alpha u, T_\beta v, a)  \big) = 0.
\end{align*}
Similarly,
\begin{align*}
&(u \triangleright_{\alpha, \beta} a) \triangleleft_{\alpha \beta, \gamma} v - u \triangleright_{\alpha, \beta \gamma} (a \triangleleft_{\beta, \gamma} v) \\
&= \big(  T_\alpha (u) \cdot a - T_{\alpha \beta} (u \cdot a) - T_{\alpha \beta} H (T_\alpha u, a)  \big) \triangleleft_{\alpha \beta, \gamma} v - u \triangleright_{\alpha, \beta \gamma} \big(  a \cdot T_\gamma (v) - T_{\beta \gamma} (a \cdot v) - T_{\beta \gamma} H (a , T_\gamma (v) )  \big) \\
&= \cancel{T_\alpha (u) \cdot a \cdot T_\gamma (v)} - T_{\alpha \beta} (u \cdot a) \cdot T_\gamma (v) - T_{\alpha \beta} (H (T_\alpha u, a)) \cdot T_\gamma (v)  - T_{\alpha \beta \gamma} (T_\alpha (u) \cdot a \cdot v) \\
& \quad + T_{\alpha \beta \gamma} (T_{\alpha \beta} (u \cdot a) \cdot v) 
+ T_{\alpha \beta \gamma} (T_{\alpha \beta} H (T_\alpha u, a) \cdot v) - T_{\alpha \beta \gamma} H (T_\alpha (u) \cdot a, T_\gamma (v)) + T_{\alpha \beta \gamma} H (T_{\alpha \beta} (u \cdot a), T_\gamma v) \\
& \quad + T_{\alpha \beta \gamma} H \big(  T_{\alpha \beta} H (T_\alpha u, a), T_\gamma v \big) - \cancel{T_\alpha (u) \cdot a \cdot T_\gamma (v)} + T_\alpha (u) \cdot T_{\beta \gamma} (a \cdot v) + T_\alpha (u) \cdot T_{\beta \gamma} H (a, T_\gamma (v) ) \\
& \quad + T_{\alpha } (u \cdot a \cdot T_\gamma (v)) - T_{\alpha \beta \gamma} (u \cdot T_{\beta \gamma} (a \cdot v)) 
- T_{\alpha \beta \gamma} (u \cdot T_{\beta \gamma} H (a, T_\gamma (v))) + T_{\alpha \beta \gamma} H (T_\alpha u, a \cdot T_\gamma (v)) \\
& \quad - T_{\alpha \beta \gamma} H (T_\alpha u, T_{\beta \gamma} (a \cdot v)) - T_{\alpha \beta \gamma} H (T_\alpha u, T_{\beta \gamma} H (a, T_\gamma v)) \\
&= T_{\alpha \beta \gamma} \bigg(  - T_{\alpha \beta} (u \cdot a) \cdot v  - u \cdot a \cdot T_\gamma (v) - H (T_{\alpha \beta} (u \cdot a), T_\gamma (v)) - T_{\alpha \beta} H (T_\alpha u, a) \cdot v  - H (T_\alpha u, a) \cdot T_\gamma (v) \\
& \quad - H (T_{\alpha \beta} H (T_\alpha u, a), T_\gamma (v)) - T_\alpha (u) \cdot a \cdot v  + T_{\alpha \beta} (u \cdot a) \cdot v + T_{\alpha \beta} H (T_\alpha u, a) \cdot v - H (T_\alpha (u) \cdot a, T_\gamma (v) ) \\ 
& \quad + H (T_{\alpha \beta} (u \cdot a), T_\gamma v) + H (T_{\alpha \beta} H (T_\alpha u, a), T_\gamma v) + T_{\alpha} (u) \cdot a \cdot v + u \cdot T_{\beta \gamma} (a \cdot v) + H (T_\alpha u, T_{\beta \gamma} (a \cdot v)) \\
& \quad + T_\alpha (u) \cdot H (a, T_\gamma v) + u \cdot T_{\beta \gamma} H (a, T_\gamma (v)) + H (T_\alpha (u), T_{\beta \gamma} H (a, T_\gamma v)) + u \cdot a \cdot T_\gamma (v) - u \cdot T_{\beta \gamma} (a \cdot v) \\
& \quad - u \cdot T_{\beta \gamma} H (a, T_\gamma (v)) + H (T_\alpha u, a \cdot T_\gamma (v)) - H (T_\alpha (u), T_{\beta \gamma} (a \cdot v)) - H (T_\alpha u, T_{\beta \gamma} H (a, T_\gamma v))\bigg) \\
&= T_{\alpha \beta \gamma} \big(  - H (T_\alpha u, a) \cdot T_\gamma (v) - H (T_\alpha (u) \cdot a, T_\gamma (v)) + T_\alpha (u) \cdot H (a, T_\gamma v) + H (T_\alpha u, a \cdot T_\gamma (v)) \big) \\
&= T_{\alpha \beta \gamma } \big(  (\delta_\mathrm{Hoch} H)(T_\alpha u, a, T_\gamma v)  \big) = 0
\end{align*}
and
\begin{align*}
&(a \triangleleft_{\alpha, \beta} u ) \triangleleft_{\alpha \beta, \gamma} v - a \triangleleft_{\alpha, \beta \gamma} (u \ast_{\beta, \gamma} v) \\
&=  \big( a \cdot T_\beta (u) - T_{\alpha \beta} (a \cdot u) - T_{\alpha \beta} H (a, T_\beta u)   \big) \triangleleft_{\alpha \beta, \gamma} v \\
&\qquad \qquad \qquad - a \cdot T_{\beta \gamma} (u \ast_{\beta, \gamma} v)  + T_{\alpha \beta \gamma} (a \cdot (u \ast_{\beta, \gamma} v)) + T_{\alpha \beta \gamma} H (a, T_{\beta \gamma} (u \ast_{\beta, \gamma} v)) \\
&= \cancel{a \cdot T_\beta (u) \cdot T_\gamma (v)} - T_{\alpha \beta} (a \cdot u) \cdot T_\gamma (v) - T_{\alpha \beta} ( H (a, T_\beta u)) \cdot T_\gamma (v) - \cancel{T_{\alpha \beta \gamma} (a \cdot T_\beta (u) \cdot v)}  \\
& \quad + T_{\alpha \beta \gamma} (T_{\alpha \beta} (a \cdot u) \cdot v) + T_{\alpha \beta \gamma} \big(  T_{\alpha \beta} ( H (a, T_\beta v)) \cdot v  \big) - T_{\alpha \beta \gamma} H (a \cdot T_\beta (u), T_\gamma (v))  + T_{\alpha \beta \gamma} H (T_{\alpha \beta} (a \cdot u), T_\gamma (v)) \\
& \quad + T_{\alpha \beta \gamma} H (T_{\alpha \beta} H (a, T_\beta u), T_\gamma (v)) - \cancel{a \cdot T_\beta (u) \cdot T_\gamma (v)} + \cancel{T_{\alpha \beta \gamma} (a \cdot T_\beta (u) \cdot v)}  \\
& \quad + T_{\alpha \beta \gamma} (a \cdot u \cdot T_\gamma (v)) + T_{\alpha \beta \gamma} (a \cdot H (T_\beta u, T_\gamma v)) + T_{\alpha \beta \gamma} H (a, T_\beta (u) \cdot T_\gamma (v)) \\
&= T_{\alpha \beta \gamma} \bigg(  - T_{\alpha \beta} (a \cdot u) \cdot v - a \cdot u \cdot T_\gamma (v) - H (T_{\alpha \beta} (a \cdot u), T_\gamma (v)) - T_{\alpha \beta} (H (a, T_\beta u)) \cdot v - H (a, T_\beta u) \cdot T_\gamma (v) \\
& \quad - H (T_{\alpha \beta} H (a, T_\beta u), T_\gamma (v)) + T_{\alpha \beta} (a \cdot u ) \cdot v + T_{\alpha \beta } (H (a, T_\beta u)) \cdot v - H (a \cdot T_\beta (u) , T_\gamma (v)) \\ 
& \quad + H (T_{\alpha \beta} (a \cdot u), T_\gamma (v)) 
+ H (T_{\alpha \beta} H (a, T_\beta u), T_\gamma v) + a \cdot u \cdot T_\gamma (v) + a \cdot H (T_\beta u, T_\gamma v) + H (a , T_\beta (u) \cdot T_\gamma (v) ) \bigg) \\
&= T_{\alpha \beta \gamma} \big(  - H (a, T_\beta u) \cdot T_\gamma (v) - H (a \cdot T_\beta (u), T_\gamma (v))  + a \cdot H (T_\beta u, T_\gamma v) + H (a , T_\beta (u) \cdot T_\gamma (v) ) \big) \\
&= T \big(  (\delta_{\mathrm{Hoch}} H) (a, T_\beta u, T_\gamma v) \big) = 0.
\end{align*}
This completes the proof.
\end{proof}

It follows that one can consider the cochain complex of the $\Omega$-associative algebra $(M, \{ ~ \ast_{\alpha, \beta} \}_{\alpha, \beta \in \Omega})$ with coefficients in the bimodule given in Theorem \ref{bim-thm}. More precisely, we define
\begin{align*}
&C^0_{\mathrm{TwOoperf}} (M,A) = A,\\
&C^{n \geq 1}_{\mathrm{TwOoperf}} (M,A) = \big\{  f = \{ f_{\alpha_1, \ldots, \alpha_n} \}_{\alpha_1, \ldots, \alpha_n \in \Omega} ~|~ f_{\alpha_1, \ldots, \alpha_n} : M^{\otimes n} \rightarrow A \text{ is multilinear} \big\}
\end{align*}
and a differential $\delta_{\mathrm{TwOoperf}} : C^n_{\mathrm{TwOoperf}} (M,A) \rightarrow C^{n+1}_{\mathrm{TwOoperf}}$ by
\begin{align}
\big( \delta_{\mathrm{TwOoperf}} (a) \big)_\alpha (u) := T_\alpha (u) \cdot a - T_\alpha (u \cdot a) - T_\alpha (H (T_\alpha u, a)) - a \cdot T_\alpha (u) + T_\alpha (a \cdot u) + T_\alpha (H (a, T_\alpha u)),
\end{align}
\begin{align}\label{tw-diff-form}
&\big( \delta_{\mathrm{TwOoperf}} (f) \big)_{\alpha_1, \ldots, \alpha_{n+1}} (u_1, \ldots, u_{n+1}) := T_{\alpha_1} (u_1) \cdot f_{\alpha_2, \ldots, \alpha_{n+1}} (u_2, \ldots, u_{n+1})  \\
& ~~- T_{\alpha_1 \cdots \alpha_{n+1}} \big( u_1 \cdot f_{\alpha_2, \ldots, \alpha_{n+1}} (u_2, \ldots, u_{n+1})  \big) - T_{\alpha_1 \cdots \alpha_{n+1}} \big( H (T_{\alpha_1} (u_1), f_{\alpha_2, \ldots, \alpha_{n+1}} (u_2, \ldots, u_{n+1}) )  \big) \nonumber \\
&~~ + \sum_{i=1}^n (-1)^i f_{\alpha_1, \ldots, \alpha_i \alpha_{i+1}, \ldots, \alpha_{n+1}} \big( u_1, \ldots, T_{\alpha_i} (u_i) \cdot u_{i+1} + u_i \cdot T_{\alpha_{i+1}} (u_{i+1}) + H (T_{\alpha_i} u_i, T_{\alpha_{i+1}} u_{i+1}), \ldots, u_{n+1} \big) \nonumber \\
&~~+ (-1)^{n+1} ~ f_{\alpha_1, \ldots, \alpha_n} (u_1, \ldots, u_n) \cdot T_{\alpha_{n+1}} (u_{n+1}) - (-1)^{n+1} T_{\alpha_1 \cdots \alpha_{n+1}} \big(  f_{\alpha_1, \ldots, \alpha_n} (u_1, \ldots, u_n) \cdot u_{n+1} \big) \nonumber \\
&~~- (-1)^{n+1} T_{\alpha_1 \cdots \alpha_{n+1}} \big( H ( f_{\alpha_1, \ldots, \alpha_n} (u_1, \ldots, u_n), T_{\alpha_{n+1}} (u_{n+1}) ) \big),  \nonumber 
\end{align}
for $a \in A = C^0_\mathrm{TwOoperf} (M,A)$ and $f \in C^{n \geq 1}_\mathrm{TwOoperf} (M,A)$.
Let $Z^n_{\mathrm{TwOoperf}} (M,A)$ be the space of $n$-cocycles and $B^n_{\mathrm{TwOoperf}} (M,A)$ be the space of $n$-coboundaries, for $n \geq 0$. The quotient groups
\begin{align*}
H^n_{\mathrm{TwOoperf}} (M, A) := \frac{ Z^n_{\mathrm{TwOoperf}} (M,A) }{ B^n_{\mathrm{TwOoperf}} (M,A) }, ~ \text{ for } n \geq 0
\end{align*}
are called the {\bf cohomology} of the $H$-twisted $\mathcal{O}$-operator family $\{ T_\alpha \}_{\alpha \in \Omega}$.

\begin{remark}
(The case $H = 0$) Let $\{ T_\alpha : M \rightarrow A \}_{\alpha \in \Omega}$ be an $\mathcal{O}$-operator family. Then $(M, \{~ \ast_{\alpha, \beta } \}_{\alpha , \beta \in \Omega} )$ is an $\Omega$-associative algebra (see \cite{agu}), where
\begin{align*}
u \ast_{\alpha, \beta} v := T_\alpha (u) \cdot v + u \cdot T_\beta (v), ~ \text{ for } u, v \in M \text{ and } \alpha, \beta \in \Omega.
\end{align*}
Moreover, $(A, \{ \triangleright_{\alpha, \beta}, \triangleleft_{\alpha, \beta} \}_{\alpha, \beta \in \Omega} )$ is a bimodule over $(M, \{~ \ast_{\alpha, \beta } \}_{\alpha , \beta \in \Omega} )$, where
\begin{align*}
u \triangleright_{\alpha, \beta} a := T_\alpha (u) \cdot a - T_{\alpha \beta} (u \cdot a) ~ \text{ and } ~ a \triangleleft_{\alpha, \beta} u := a \cdot T_\beta (u) - T_{\alpha \beta} (a \cdot u), \text{ for } u \in M, a \in A.
\end{align*}
The corresponding cohomology groups are called the cohomology of the $\mathcal{O}$-operator family $\{ T_\alpha \}_{\alpha \in \Omega}$.
\end{remark}


\medskip

\noindent $\square$ {\bf Deformation theory of $H$-twisted $\mathcal{O}$-operator family.}

\medskip

Here we consider formal deformations of an $H$-twisted $\mathcal{O}$-operator family. We show that the cohomology of $H$-twisted $\mathcal{O}$-operator family introduced above govern such deformations.

Let $A$ be an associative algebra, $M$ be an $A$-bimodule and $H: A^{\otimes 2} \rightarrow M$ be a $2$-cocycle. Consider the space $A[[t]]$ of formal power series in $t$ with coefficients from $A$. Then the associative multiplication on $A$ induces an associative multiplication on $A[[t]]$ using ${\bf k}[[t]]$-bilinearity. Further, the space $M[[t]]$ is an $A[[t]]$-bimodule. The $2$-cocycle $H$ can also be extended to a $2$-cocycle (denoted by the same notation) $H : A[[t]]^{\otimes 2} \rightarrow M[[t]]$ using ${\bf k}[[t]]$-bilinearity.

\begin{defn}
Let $T = \{ T_\alpha : M \rightarrow A \}_{\alpha \in \Omega}$ be an $H$-twisted $\mathcal{O}$-operator family. A {\bf formal one-parameter deformation} of $T$ consists of a formal sum
\begin{align*}
T^t := \sum_{i=0}^\infty~ t^i~ T^{(i)} , \text{ where } T^{(i)} = \{ T^{(i)}_\alpha : M \rightarrow A \}_{\alpha \in \Omega} \text{ with } T^{(0)} = T
\end{align*}
such that the ${\bf k}[[t]]$-linear map $T^{(t)} : M[[t]] \rightarrow A[[t]]$ is an $H$-twisted $\mathcal{O}$-operator family. That is,
\begin{align*}
T^t_\alpha (u) \cdot T^t_\beta (v) = T^t_{\alpha \beta} \big( T^t_\alpha (u) \cdot v + u \cdot T^t_\beta (v) + H (T^t_\alpha (u), T^t_\alpha (v) )   \big), \text{ for } u, v \in M \text{ and } \alpha, \beta \in \Omega.
\end{align*}
\end{defn}

It follows that $T^t$ is a formal one-parameter deformation of $T$ if 
\begin{align}\label{defor-eqn}
\sum_{i+j=n} T^{(i)}_\alpha (u) \cdot T^{(j)}_\beta (v) = \sum_{i+j = n} T_{\alpha \beta}^{(i)} \big(  T_\alpha^{(j)} (u) \cdot v + u \cdot T_\beta^{(j)}(v)  \big) + \sum_{i+j+k = n} T_{\alpha \beta}^{(i)} ~\big(  H ( T_\alpha^{(j)} (u), T_\beta^{(k)} (v) ) \big)
\end{align}
holds, for $n = 0, 1, \ldots .$ These system of equations are called deformation equations. Note that (\ref{defor-eqn}) holds for $n=0$ as $T^{(0)} = T$ is an $H$-twisted $\mathcal{O}$-operator family. For $n=1$, we obtain
\begin{align*}
T^{(1)}_\alpha (u) \cdot T_\beta (v) + T_\alpha (u) \cdot T^{(1)}_\beta (v) =~& T^{(1)}_{\alpha \beta} \big( T_\alpha (u) \cdot v + u \cdot T_\beta (v) + H (T_\alpha (u), T_\beta (v) )  \big) \\
+& T_{\alpha \beta} \big(  T^{(1)}_\alpha (u) \cdot v + u \cdot T^{(1)}_\beta (v) + H (   T^{(1)}_\alpha (u), T_\beta (v))  + H ( T_\alpha (u), T_\beta^{(1)}(v))  \big).
\end{align*}
It follows from (\ref{tw-diff-form}) that the above identity is equivalent to $\big( \delta_\mathrm{TwOoperf} (T^{(1)}) \big)_{\alpha, \beta} (u, v) = 0$, for all $u, v \in M$ and $\alpha, \beta \in \Omega$. In other words, $T^{(1)} = \{ T^{(1)}_\alpha \}_{\alpha \in \Omega} \in C^1_\mathrm{TwOoperf} (M,A)$ is a $1$-cocycle in the cochain complex of the $H$-twisted $\mathcal{O}$-operator family $T = \{ T_\alpha \}_{\alpha \in \Omega}$. The $1$-cocycle $T^{(1)}$ is called the {\bf infinitesimal} of the deformation $T^{t}$.

\begin{defn}
Two formal one-parameter deformations $T_t = \sum_{i=0}^\infty t^i ~ T^{(i)}$ and $\overline{T}_t = \sum_{i=0}^\infty t^i~ \overline{T}^{(i)}$ of the $H$-twisted $\mathcal{O}$-operator family $T= \{ T_\alpha \}_{\alpha \in \Omega}$ are said to be {\bf equivalent} if there exists an element $\theta \in A$, linear maps $\phi_i \in \mathrm{Hom}(A,A)$ and $\psi_i \in \mathrm{Hom}(M,M)$, for $i \geq 2$, such that the maps
\begin{align*}
&\phi^t = \mathrm{id}_A + t (l^\mathrm{ad}_\theta - r^\mathrm{ad}_\theta) + \sum_{i=2}^\infty t^i \phi_i : ~A[[t]] \rightarrow A[[t]] \\
& \big\{ \psi^t_\alpha = \mathrm{id}_M + t (l_\theta - r_\theta + H (\theta, T_\alpha -) - H (T_\alpha -, \theta) ) + \sum_{i=2}^\infty t^i \psi^i  : ~ M[[t]] \rightarrow M[[t]]  \big\}_{\alpha \in \Omega}
\end{align*}
satisfy $\phi^t \circ T^t_\alpha = \overline{T}^t_\alpha \circ \psi^t_\alpha$, for all $\alpha \in \Omega$.
\end{defn}



By equating coefficients of $t$ in both sides of $\phi^t \circ T^t_\alpha (u)= \overline{T}^t_\alpha \circ \psi^t_\alpha (u)$, we get
\begin{align*}
T_\alpha^{(1)} (u) + \theta \cdot T_\alpha (u) - T_\alpha (u) \cdot \theta  = \overline{T}_\alpha (u) + T_\alpha \big(   \theta \cdot u - u \cdot \theta + H (\theta, T_\alpha u) - H (T_\alpha u, \theta) \big).
\end{align*}
This is equivalent to
\begin{align*}
\big( T^{(1)}   - \overline{T}^{(1)} \big)_\alpha (u) =  \big( \delta_\mathrm{TwOoperf} (\theta) \big)_\alpha (u).
\end{align*}

As a consequence, we get the following.

\begin{thm}\label{inf-1-co}
Let $T^t = \sum_{i=0}^\infty t^i T^{(i)}$ be a formal one-parameter deformation of the $H$-twisted $\mathcal{O}$-operator family $T = \{ T_\alpha \}_{\alpha \in \Omega}$. Then the infinitesimal $T^{(1)}$ is a $1$-cocycle in the cochain complex of $T = \{ T_\alpha \}_{\alpha \in \Omega}$. Moreover, the corresponding cohomology class depends only on the equivalence class of the deformation $T^t$.
\end{thm}

One may also consider rigidity of an $H$-twisted $\mathcal{O}$-operator family. An $H$-twisted $\mathcal{O}$-operator family $T = \{ T_\alpha \}_{\alpha \in \Omega}$ is said to be {\bf rigid} if any formal one-parameter deformation $T^t = \sum_{i=0}^\infty t^i T^{(i)}$ of $T$ is equivalent to the undeformed one $\overline{T}^t = T$.



\begin{thm}
Let $T = \{ T_\alpha \}_{\alpha \in \Omega}$ be an $H$-twisted $\mathcal{O}$-operator family. If $H^1_\mathrm{TwOoperf} (M,A) = 0$ then $\{ T_\alpha \}_{\alpha \in \Omega}$ is rigid.
\end{thm}

\begin{proof}
Let $T^t = \sum_{i=0}^\infty t^i T^{(i)}$ be any formal deformation of $T$. Then by Theorem \ref{inf-1-co}, the infinitesimal $T^{(1)} \in Z^1_\mathrm{TwOoperf} (M,A)$ is a $1$-cocycle. Hence by the assumption, there exists $\theta \in A = C^0_\mathrm{TwOoperf}(M,A)$ such that $T^{(1)} = \delta_\mathrm{TwOoperf} (\theta).$ We set
\begin{align*}
\phi^t := \mathrm{id}_A + t (l^\mathrm{ad}_\theta - r^\mathrm{ad}_\theta) ~ \text{ and } ~ \psi^t_\alpha := \mathrm{id}_M + t \big(  l_\theta - r_\theta + H (\theta, T_\alpha -) - H (T_\alpha - , \theta)  \big), \text{ for } \alpha \in \Omega 
\end{align*}
and define $\overline{T}^t_\alpha = \phi_t \circ T^t_\alpha \circ (\psi^t_\alpha)^{-1}$. Then $T^t$ is equivalent to the deformation $\overline{T}^t$. Moreover, we observe that
\begin{align*}
&\overline{T}^t_\alpha (u) \quad (\mathrm{mod } t^2) \\
&= \big(    \mathrm{id}_A + t (l^\mathrm{ad}_\theta - r^\mathrm{ad}_\theta)\big) \circ \big(  T_\alpha + t T^{(1)}_\alpha \big)  \big(   u - t \big( \theta \cdot u - u \cdot \theta + H (\theta, T_\alpha u) - H (T_\alpha u, \theta)   \big) \big)  \quad (\mathrm{mod } t^2)  \\
&= T_\alpha (u) + t ~ \big( T_\alpha^{(1)} u - T_\alpha \big(       \theta \cdot u - u \cdot \theta + H (\theta, T_\alpha u) - H (T_\alpha u, \theta)  \big) + \theta \cdot T_\alpha (u) - T_\alpha (u) \cdot \theta    \big) \\
&= T_\alpha (u) ~~~~ (\text{as } T^{(1)} = \delta_\mathrm{TwOoperf} (\theta)\big).
\end{align*}
This shows that the coefficient of $t$ in the formal expression of $\overline{T}^t$ vanishes. By repeating this process, we can show that the deformation $T^t$ is equivalent to the undeformed one.
\end{proof}

\subsection{Cohomology and deformations of NS-family algebras}
In this subsection, we define the cohomology of NS-family algebras. In particular, we get the cohomology of dendriform family algebras. Some remarks about deformations of these family algebras are also made. Throughout this subsection, $\Omega$ is a semigroup (not necessarily with the unit).

Let $(D, \{ \prec_\alpha, \succ_\alpha, \curlyvee_{\alpha, \beta} \}_{\alpha, \beta \in \Omega})$ be a NS-family algebra. To define the cohomology of this NS-family algebra, we need some notations. For each $n \geq 1$, we define a set $C_n$ as follows
\begin{align*}
C_1 = \{ [1] \} ~~ \text{ and } ~~ C_{n \geq 2}  = \{ [1], [2], \ldots, [n+1] \}.
\end{align*}
In other words, $C_1$ is the set consisting of the natural number $1$ and $C_{n \geq 2}$ is the set consisting of the first $n+1$ natural numbers (all putted inside the square brackets). For any $m , n \geq 1$ and $1 \leq i \leq m$, we define maps
\begin{align*}
R_{m;i, n} : C_{m+n-1} \setminus \{ [m+n] \} \rightarrow C_m ~~~ \text{ and } ~~~ S_{m;i,n} : C_{m+n-1} \setminus \{ [m+n] \} \rightarrow {\bf k}[C_n]
\end{align*} 
by
\begin{align*}
R_{m;i, n} ([r]) = \begin{cases}
[r] ~& \text{ if } 1 \leq r \leq i-1 \\
[i] ~& \text{ if } i \leq r \leq i+n-1 \\
[r-n+1] ~& \text{ if } i+n \leq r \leq m+n-1,
\end{cases}\\
S_{m;i, n} ([r]) = \begin{cases}
\sum_{[s] \in C_n} [s]  ~& \text{ if } 1 \leq r \leq i-1 \\
[r-i+1] ~& \text{ if } i \leq r \leq i+n-1 \\
\sum_{[s] \in C_n} [s]  ~& \text{ if } i+n \leq r \leq m+n-1.
\end{cases}
\end{align*}
We now consider the space $\mathrm{Hom}_{\Omega^{\times n}} ({\bf k}[C_n] \otimes D^{\otimes n}, D)$ whose elements are tuples $f = (f^{[1]}, \ldots, f^{[n+1]})$ when $n \geq 2$ (and $f = f^{[1]}$ when $n=1$), where each $f^{[i]}$ is a collection 
\begin{align*}
f^{[i]} = \{ f^{[i]}_{\alpha_1, \ldots, \alpha_n} : D^{\otimes n} \rightarrow D \}_{\alpha_1, \ldots, \alpha_n \in \Omega}
\end{align*}
of $n$-ary multilinear maps on $D$ labelled by the elements of $\Omega^{\times n}$. For each $n \geq 1$, we define the $n$-th cochain group $C^n_{\mathrm{NSfam}} (D,D)$ by
\begin{align*}
C^n_\mathrm{NSfam} (D,D) = \bigg\{  f \in \mathrm{Hom}_{\Omega^{\times n}} ({\bf k}[C_n] \otimes D^{\otimes n}, D) ~ \bigg| ~ f^{[i]} \text{ doesnot depend on } \alpha_i, \text{ for } [i] \in C_n \setminus \{ [n+1] \} \bigg\}.
\end{align*}

\begin{remark}\label{fam-pi}
Note that the NS-family algebra structure $(D, \{ \prec_\alpha , \succ_\alpha, \curlyvee_{\alpha, \beta} \}_{\alpha, \beta \in \Omega})$ induces an element $\pi = (\pi^{[1]}, \pi^{[2]}, \pi^{[3]}) \in C^2_\mathrm{NSfam}(D,D)$ given by
\begin{align*}
\pi^{[1]} =~& \{ \pi^{[1]}_{\alpha, \beta} := \prec_\beta~ : D^{\otimes 2} \rightarrow D \}_{\alpha, \beta \in \Omega}, \\
\pi^{[2]} =~& \{ \pi^{[2]}_{\alpha, \beta} := \succ_\alpha~ : D^{\otimes 2} \rightarrow D \}_{\alpha, \beta \in \Omega}, \\
\pi^{[3]} =~& \{ \pi_{\alpha, \beta}^{[3]} := \curlyvee_{\alpha, \beta} ~: D^{\otimes 2} \rightarrow D \}_{\alpha, \beta \in \Omega}.
\end{align*}
The element $\pi$ will be crucial to define the coboundary operator of the NS-family algebra.
\end{remark}

We define a map $\delta_\mathrm{NSfam} : C^n_\mathrm{NSfam} (D,D) \rightarrow C^{n+1}_\mathrm{NSfam}(D,D)$ by
\begin{align*}
&\big(  \delta_\mathrm{NSfam} (f) \big)^{[r]}_{\alpha_1, \ldots, \alpha_{n+1}} (x_1, \ldots, x_{n+1}) \\
&:= \pi^{R_{2;1,n} ([r])}_{\alpha_1, \alpha_2 \cdots \alpha_{n+1}} \big( x_1, ~f^{S_{2;1,n} ([r])}_{\alpha_2, \ldots, \alpha_{n+1}} (x_2, \ldots, x_{n+1}) \big) \\
& \quad +\sum_{i=1}^n (-1)^i ~ f^{R_{n;i,2} ([r])}_{\alpha_1, \ldots, \alpha_i \alpha_{i+1}, \ldots, \alpha_{n+1}} \big( x_1, \ldots, x_{i-1}, \pi^{S_{n;i,2} ([r])}_{\alpha_i, \alpha_{i+1}} (x_i, x_{i+1}), x_{i+2}, \ldots, x_{n+1}  \big) \\
& \quad + (-1)^{n+1} \pi^{R_{2;n,1} ([r])}_{\alpha_1 \cdots \alpha_n, \alpha_{n+1}} \big(  f^{S_{2;n,1} ([r])}_{\alpha_1, \ldots, \alpha_n} (x_1, \ldots, x_n), x_{n+1}  \big) \qquad (\text{when } [r] \in C_{n+1} \setminus \{ [n+2]\})
\end{align*}
and
\begin{align*}
&\big(  \delta_\mathrm{NSfam} (f) \big)^{[n+2]}_{\alpha_1, \ldots, \alpha_{n+1}} (x_1, \ldots, x_{n+1}) \\
&:= \pi^{[2]}_{\alpha_1, \alpha_2 \cdots \alpha_{n+1}} \big( x_1, ~f^{[n+1]}_{\alpha_2, \ldots, \alpha_{n+1}} (x_2, \ldots, x_{n+1}) \big) ~+~ \pi^{[3]}_{\alpha_1, \alpha_2 \cdots \alpha_{n+1}} \big( x_1, ~f^{[1]+ \cdots +[n+1]}_{\alpha_2, \ldots, \alpha_{n+1}} (x_2, \ldots, x_{n+1}) \big)\\
&\quad +\sum_{i=1}^n (-1)^i ~ f^{[i]}_{\alpha_1, \ldots, \alpha_i \alpha_{i+1}, \ldots, \alpha_{n+1}} \big( x_1, \ldots, x_{i-1}, \pi^{[3]}_{\alpha_i, \alpha_{i+1}} (x_i, x_{i+1}), x_{i+2}, \ldots, x_{n+1}  \big) \\
& \quad +\sum_{i=1}^n (-1)^i ~ f^{[n+1]}_{\alpha_1, \ldots, \alpha_i \alpha_{i+1}, \ldots, \alpha_{n+1}} \big( x_1, \ldots, x_{i-1}, \pi^{[1]+[2]+[3]}_{\alpha_i, \alpha_{i+1}} (x_i, x_{i+1}), x_{i+2}, \ldots, x_{n+1}  \big) \\
& \quad + (-1)^{n+1}~ \big\{ \pi^{[1]}_{\alpha_1 \cdots \alpha_n, \alpha_{n+1}} \big(  f^{[n+1]}_{\alpha_1, \ldots, \alpha_n} (x_1, \ldots, x_n), x_{n+1}  \big)  + \pi^{[3]}_{\alpha_1 \cdots \alpha_n, \alpha_{n+1}} \big(  f^{[1]+ \cdots +[n+1]}_{\alpha_1, \ldots, \alpha_n} (x_1, \ldots, x_n), x_{n+1}  \big) \big\},
\end{align*}
for $f \in C^n_\mathrm{NSfam} (D,D)$, $\alpha_1, \ldots, \alpha_{n+1} \in \Omega$ and $x_1, \ldots, x_{n+1} \in D$. Then similar to the coboundary operator for ordinary NS-algebras \cite{An}, one can show that $(\delta_\mathrm{NSfam})^2 = 0$. In other words, $\{ C^{n \geq 1}_\mathrm{NSfam} (D,D), \delta_\mathrm{NSfam} \}$ is a cochain complex. The corresponding cohomology groups are called the cohomology of the NS-family algebra $(D, \{ \prec_\alpha, \succ_\alpha, \curlyvee_{\alpha, \beta} \}_{\alpha, \beta \in \Omega})$, and they are denoted by $H^{n \geq 1}_\mathrm{NSfam}(D,D)$.

\medskip

\noindent $\square$ {\bf Deformation theory of NS-family algebras.}

\medskip

Let $(D, \{ \prec_\alpha, \succ_\alpha, \curlyvee_{\alpha, \beta} \}_{\alpha, \beta \in \Omega})$ be an NS-family algebra. Consider the element $\pi = (\pi^{[1]}, \pi^{[2]}, \pi^{[3]}) \in C^2_\mathrm{NSfam}(D,D)$ defined in Remark \ref{fam-pi}. A formal one-parameter deformation of this NS-family algebra consist of formal sums
\begin{align}
\prec_\alpha^t ~=~ \prec^0_\alpha +~& t \prec^1_\alpha +~ t^2 \prec^2_\alpha + \cdots  , \qquad  \succ_\alpha^t ~=~ \succ^0_\alpha +~ t \succ^1_\alpha +~ t^2 \succ^2_\alpha + \cdots , \label{form-nsns}\\
&\curlyvee_{\alpha, \beta}^t =~ \curlyvee^0_{\alpha, \beta} +~ t \curlyvee^1_{\alpha, \beta} +~ t^2 \curlyvee^2_{\alpha,\beta} + \cdots , \text{ for } \alpha, \beta \in \Omega \label{form-nsnsns}
\end{align}
with $\prec_\alpha^0 = \prec_\alpha$, $\succ_\alpha^0 = \succ_\alpha$ and $\curlyvee_{\alpha, \beta}^0 = \curlyvee_{\alpha, \beta}$ that makes $(D[[t]], \{ \prec_\alpha^t, \succ_\alpha^t, \curlyvee_{\alpha, \beta}^t \}_{\alpha, \beta \in \Omega})$ into an NS-family algebra (over the base ring ${\bf k}[[t]]$.)

The existence of the formal sums (\ref{form-nsns}), (\ref{form-nsnsns}) is equivalent to having a formal sum
\begin{align*}
\pi^t = \pi^0 +~ t \pi^1 + t^2 \pi^2 + \cdots \in C^2_\mathrm{NSfam}(D,D)[[t]] \text{ with } \pi^0 = \pi.
\end{align*}
By comparing coefficients of various powers of $t$ in the NS-family identities for $(D[[t]], \{ \prec_\alpha^t, \succ_\alpha^t, \curlyvee_{\alpha, \beta}^t \}_{\alpha, \beta \in \Omega})$, we get the deformation equations. We may also interpret these deformation equations in terms of the elements $\pi^i \in C^2_\mathrm{NSfam}(D,D)$, for $i \geq 0$. Comparing coefficients of $t$, we simply get
\begin{align*}
\delta_\mathrm{NSfam} (\pi^1) = 0.
\end{align*}
The element $\pi^1$ is called the infinitesimal of the formal one-parameter deformation.

Two formal one-parameter deformations $\{ \prec_\alpha^t, \succ_\alpha^t, \curlyvee_{\alpha, \beta}^t \}_{\alpha, \beta \in \Omega}$ and $\{ \overline{\prec}_\alpha^t, \overline{\succ}_\alpha^t, \overline{\curlyvee}_{\alpha, \beta}^t \}_{\alpha, \beta \in \Omega}$ are said to be equivalent if there exist linear maps $\psi^i : D \rightarrow D$ for $i \geq 1$ such that the formal sum
\begin{align*}
\psi^t = \mathrm{id}_D + t \psi^1 + t^2 \psi^2 + \cdots  ~: D[[t]] \rightarrow D[[t]]
\end{align*}
defines a morphism of NS-family algebras from $(D[[t]], \{ \prec_\alpha^t, \succ_\alpha^t, \curlyvee_{\alpha, \beta}^t \}_{\alpha, \beta \in \Omega})$ to $(D[[t]], \{ \overline{\prec}_\alpha^t, \overline{\succ}_\alpha^t, \overline{\curlyvee}_{\alpha, \beta}^t \}_{\alpha, \beta \in \Omega})$. The morphism condition for $\psi^t$ also leads to a system of equations by comparing various powers of $t$. In particular, by comparing coefficients of $t$, we get  (using the elements $\pi^i$ and $\overline{\pi}^i$ corresponding to two deformations)
\begin{align*}
\pi^1 - \overline{\pi}^1 = \delta_\mathrm{NSfam}(\psi^1).
\end{align*}

As a summary, we get the following.

\begin{thm}\label{final-thmm}
Let $D = (D, \{ \prec_\alpha, \succ_\alpha , \curlyvee_{\alpha, \beta} \}_{\alpha, \beta \in \Omega})$ be an NS-family algebra. Then the infinitesimal in any formal deformation is a $2$-cocycle in the cohomology complex of $D$. Moreover, the corresponding cohomology class depends only on the equivalence class of the deformation.
\end{thm}

\begin{remark} ({\bf Cohomology of dendriform family algebras})
Let $(D, \{ \prec_\alpha, \succ_\alpha \}_{\alpha \in \Omega})$ be a dendriform family algebra. It can be considered as an NS-family algebra $(D, \{ \prec_\alpha, \succ_\alpha , \curlyvee_{\alpha, \beta} \}_{\alpha, \beta \in \Omega})$ in which $\curlyvee_{\alpha, \beta} = 0$ for all $\alpha, \beta \in \Omega$. Hence one can define a cochain complex $\{ C^{n \geq 1}_{\mathrm{NSfam}} (D,D), \delta_\mathrm{NSfam} \}$ as above. There is a subcomplex $$C^{n \geq 1}_{\mathrm{Dendfam}} (D,D) \hookrightarrow C^{n \geq 1}_{\mathrm{NSfam}} (D,D) ~~\text{ given by }$$
\begin{align*}
C^1_\mathrm{Dendfam}(D,D) = C^1_\mathrm{NSfam}(D,D) ~~ \text{ and } ~~ C^{n \geq 2}_\mathrm{Dendfam} (D,D) = \{ f \in C^n_\mathrm{NSfam}(D,D) ~|~ f^{[n+1]} = 0 \}.
\end{align*}
It can be easily verified that the differential $\delta_\mathrm{NSfam}$ restricts to a differential (which we denote by $\delta_\mathrm{Dendfam}$) on $C^{n \geq 1}_{\mathrm{Dendfam}} (D,D)$. In other words, $\{ C^{n \geq 1}_{\mathrm{Dendfam}} (D,D), \delta_\mathrm{Dendfam} \}$ is a cochain complex. The cohomology of this cochain complex is called the cohomology of the dendriform family algebra $(D, \{ \prec_\alpha, \succ_\alpha \}_{\alpha \in \Omega})$. When $\Omega$ is singleton (i.e. the dendriform family algebra is nothing but an ordinary dendriform algebra), this cohomology coincides with the cohomology of ordinary dendriform algebra \cite{lod-val-book} (see also \cite{A4}).

One may also define formal deformations of a dendriform family algebra and equivalences of formal deformations. In particular, Theorem \ref{final-thmm} follows in the context of dendriform family algebras.
\end{remark}

\section{Appendix}
Many results of the present paper can be dualized into coassociative coalgebras. In this appendix, we only present the dual versions of the results of Section \ref{sec-3}. 

\medskip

\noindent $\square$ {\bf Twisted $\mathcal{O}$-operator cofamilies.}

\medskip

Let $(C, \triangle)$ be a coassociative coalgebra, i.e., $C$ is a vector space and $\triangle : C \rightarrow C \otimes C$ is a linear map satisfying the coassociativity $(\triangle \otimes \mathrm{id}) \circ \triangle = (\mathrm{id} \otimes \triangle) \circ \triangle$. A $C$-cobimodule is a vector space $N$ together with linear maps $\triangle^l : N \rightarrow C \otimes N$ and $\triangle^r : N \rightarrow N \otimes C$ (called the left and right $C$-coactions, respectively) that satisfies
\begin{align*}
(\triangle \otimes \id) \circ \triangle^{\ell} = (\id \otimes \triangle^{\ell}) \circ \triangle^{\ell},~~(\triangle^{\ell} \otimes \id) \circ \triangle^{r} = (\id \otimes \triangle^{r}) \circ \triangle^{\ell}~~\text{and}~~(\triangle^{r} \otimes \id) \circ \triangle^{r} = (\id \otimes \triangle) \circ \triangle^{r}.
\end{align*}
It follows that the coassociative coalgebra $(C, \triangle)$ is a $C$-cobimodule with $\triangle^l = \triangle^r = \triangle.$

Let $(C, \triangle)$ be a coassociative coalgebra and $N = (N, \triangle^l , \triangle^r)$ be a $C$-cobimodule. The corresponding cohomology theory (called the coHochschild cohomology theory, also known as the Cartier cohomology theory) is dual to the standard Hochschild cohomology. A linear map $h : N \rightarrow C \otimes C$ is called a coHochschild $2$-cocycle if
\begin{align*}
(\mathrm{id} \otimes h) \circ \triangle^l ~-~ (\triangle \otimes \mathrm{id}) \circ h ~+~ (\mathrm{id} \otimes \triangle) \circ h ~-~ (h \otimes \mathrm{id}) \circ \triangle^r = 0.
\end{align*}

With the above notations, we can define the following which is dual to Definition \ref{defn-tttt}.

\begin{defn}
An {\bf $h$-twisted $\mathcal{O}$-operator cofamily} is a collection $\{ S_\alpha : C \rightarrow N \}_{\alpha \in \Omega}$ of linear maps satisfying for $\alpha , \beta \in \Omega$,
\begin{align*}
(S_\alpha \otimes S_\beta) \circ \triangle = \big(  (S_\alpha \otimes \mathrm{id}) \circ \triangle^l ~+~ (\mathrm{id} \otimes S_\beta) \circ \triangle^r ~+~ (S_\alpha \otimes S_\beta) \circ h \big) \circ S_{\alpha \beta}.
\end{align*}
\end{defn}
When $h = 0$, we call it an $\mathcal{O}$-operator family. Moreover, if $N = C$, we call it a Rota-Baxter cofamily.

Let $(C, \triangle)$ be a coassociative coalgebra. Then the dual vector space $C^{dual} = \mathrm{Hom}(C, {\bf k})$ carries an associative algebra structure with the multiplication given by
\begin{align*}
C^{dual} \otimes C^{dual} \xrightarrow{I_C} (C \otimes C)^{dual} \xrightarrow{\triangle^{dual}} C^{dual},
\end{align*}
where $I_C : C^{dual} \otimes C^{dual} \rightarrow (C \otimes C)^{dual}$ is the map $I_C (f \otimes g)(c \otimes d) =f(c) ~g(d)$, for $f \otimes g \in C^{dual} \otimes C^{dual}$ and $c \otimes d \in C \otimes C$. Moreover, if $N = (N, \triangle^l, \triangle^r)$ is a $C$-cobimodule then the dual space $N^{dual}$ can be equipped with a $C^{dual}$-bimodule structure. Further, if $h : N \rightarrow C \otimes C$ is a coHochschild $2$-cocycle then 
\begin{align*}
H : C^{dual} \otimes C^{dual} \xrightarrow{I_C} (C \otimes C)^{dual} \xrightarrow{h^{dual}} N^{dual}
\end{align*}
is a Hochschild $2$-cocycle. With these notations, we have the following.

\begin{prop}
Let $\{ S_\alpha : C \rightarrow N \}_{\alpha \in \Omega}$ be an $h$-twisted $\mathcal{O}$-operator cofamily. Then the collection 
\begin{center}
$\{ S^{dual}_\alpha : N^{dual} \rightarrow C^{dual}, S^{dual}_\alpha (f) (c) = f (S_\alpha (c)) \}_{\alpha \in \Omega}$ is an $H$-twisted $\mathcal{O}$-operator family.
\end{center}
\end{prop}

\begin{proof}
For any $f, g \in N^{dual}$, $\alpha , \beta \in \Omega$ and $c \in N$, we have
\begin{align*}
\big( S_\alpha^{dual} (f) \cdot S_\beta^{dual} (g) \big) (c) =~& \sum S_\alpha^{dual} (f) (c_{(1)}) ~ S_\beta^{dual} (g) (c_{(2)}) \qquad (\text{assuming }~ \triangle (c) = \sum c_{(1)} \otimes c_{(2)}) \\
=~& \sum f \big( S_\alpha (c_{(1)}) \big) ~ g \big( S_\beta (c_{(2)})  \big) \\
=~& \sum I_N (f \otimes g) \big(  S_\alpha (c_{(1)}) \otimes   S_\beta (c_{(2)}) \big) \\
=~& \sum I_N (f \otimes g) \big(  (S_\alpha \otimes \mathrm{id}) \circ \triangle^l ~+~ (\mathrm{id} \otimes S_\beta) \circ \triangle^r ~+~ (S_\alpha \otimes S_\beta) \circ h   \big) (S_{\alpha \beta} (c)) \\
=~& \big(  S_\alpha^{dual} (f) \cdot g ~+~ f \cdot S_\beta^{dual}(g) ~+~ H (S_\alpha^{dual}(f), S_\beta^{dual} (g))  \big) 
 (S_{\alpha \beta} (c)) \\
 =~& S^{dual} \big(  S_\alpha^{dual} (f) \cdot g ~+~ f \cdot S_\beta^{dual}(g) ~+~ H (S_\alpha^{dual}(f), S_\beta^{dual} (g))  \big) (c).
\end{align*}
This proves the result.
\end{proof}

\medskip

\noindent $\square$ {\bf NS-cofamily coalgebras.}


\begin{defn}
An {\bf NS-cofamily coalgebra} is a pair $(C, \{ \triangle_{\prec_\alpha}, \triangle_{\succ_\alpha}, \triangle_{\curlyvee_{\alpha, \beta}} \}_{\alpha, \beta \in \Omega})$ consisting of a vector space $C$ together with a collection $\{ \triangle_{\prec_\alpha}, \triangle_{\succ_\alpha}, \triangle_{\curlyvee_{\alpha, \beta}} \}_{\alpha, \beta \in \Omega}$ of linear maps satisfying for $\alpha, \beta, \gamma \in \Omega$,
\begin{align}
(\triangle_{\prec_\alpha} \otimes \id) \circ \triangle_{\prec_{\beta}}                                    &= ( \id \otimes (\triangle_{\succ_\alpha} + \triangle_{\prec_\beta}  +   \triangle_{\curlyvee_{\alpha, \beta}}) ) \circ \triangle_{\prec_{\alpha \beta}},\\
(\triangle_{\succ_{\alpha}} \otimes \id) \circ \triangle_{\prec_{\beta}}                                  &= (\id \otimes \triangle_{\prec_{\beta}}) \circ \triangle_{\succ_{\alpha}},\\
((\triangle_{\succ_\alpha} + \triangle_{\prec_\beta} + \triangle_{\curlyvee_{\alpha, \beta}}) \otimes \id) \circ \triangle_{\succ_{\alpha \beta}} &= (\id \otimes \triangle_{\succ_\beta}) \circ \triangle_{\succ_{\alpha}} ,\\
((\triangle_{\succ_\alpha} + \triangle_{\prec_\beta} + \triangle_{\curlyvee_{\alpha, \beta}}) &\otimes \id) \circ \triangle_{\curlyvee_{\alpha \beta, \gamma}} ~+~ (\triangle_{\curlyvee_{\alpha, \beta}} \otimes \mathrm{id}) \circ \triangle_{\prec_\gamma} \\
&= (\mathrm{id} \otimes \triangle_{\curlyvee_{\beta, \gamma}}) \circ \triangle_{\succ_\alpha} ~+~ (\mathrm{id} \otimes (  \triangle_{\succ_\beta} + \triangle_{\prec_\gamma} + \triangle_{\curlyvee_{\beta, \gamma}}  ) ) \circ \triangle_{\curlyvee_{\alpha, \beta \gamma}}. \nonumber
\end{align}
\end{defn}

When $\triangle_{\curlyvee_{\alpha, \beta}} = 0$ for all $\alpha, \beta \in \Omega$, then the pair $(C, \{ \triangle_{\prec_\alpha}, \triangle_{\succ_\alpha})$ is called a dendriform-cofamily coalgebra.


The proof of the following result is dual to the proof of Proposition \ref{twisted-rota-ns}. Hence we will not repeat it here.

\begin{prop}
Let $C$ be a coassociative coalgebra, $N = (N, \triangle^l, \triangle^r)$ be a $C$-cobimodule and $h: N \rightarrow C \otimes C$ be a coHochschild $2$-cocycle. Suppose $\{ S_\alpha : C \rightarrow N \}_{\alpha \in \Omega}$ be an $h$-twisted $\mathcal{O}$-operator cofamily. Then the pair $(N, \{ \triangle_{\prec_\alpha}, \triangle_{\succ_\alpha}, \triangle_{\curlyvee_{\alpha, \beta}} \}_{\alpha, \beta \in \Omega})$ is a NS-cofamily coalgebra, where
\begin{align*}
\triangle_{\prec_\alpha} = (\mathrm{id} \otimes S_\alpha) \circ \triangle^r, ~~~~~ 
\triangle_{\succ_\alpha} = ( S_\alpha \otimes \mathrm{id}) \circ \triangle^l ~~ \text{ and } ~~ \triangle_{\curlyvee_{\alpha, \beta}} = (S_\alpha \otimes S_\beta) \circ h, ~\text{for } \alpha, \beta \in \Omega. 
\end{align*}
In particular, if $\{ S_\alpha : C \rightarrow N \}_{\alpha \in \Omega}$ is an $\mathcal{O}$-operator cofamily then $(N, \{ \triangle_{\prec_\alpha}, \triangle_{\succ_\alpha} \}_{\alpha \in \Omega})$ is a dendriform-cofamily coalgebra, where
\begin{align*}
\triangle_{\prec_\alpha} = (\mathrm{id} \otimes S_\alpha) \circ \triangle^r ~~ \text{ and } ~~
\triangle_{\succ_\alpha} = ( S_\alpha \otimes \mathrm{id}) \circ \triangle^l, ~\text{for } \alpha \in \Omega. 
\end{align*}
\end{prop}

\begin{remark}
In Section \ref{sec-5}, we defined cohomology theories for twisted $\mathcal{O}$-operator family and NS-family algebra. By dualizing such cohomologies, one can easily define cohomologies for twisted $\mathcal{O}$-operator cofamily and NS-cofamily coalgebra. Such cohomologies govern respective deformations.
\end{remark}

\noindent {\bf Acknowledgements.} The author would like to thank Dr Sourav Sen for some useful discussions and support regarding this work. He also wishes to thank Indian Institute of Technology (IIT) Kharagpur for providing the beautiful academic atmosphere where the research has been carried out.


\begin{thebibliography}{BFGM03}

\bibitem{aguiar} M. Aguiar, Pre-Poisson algebras, {\em Lett. Math. Phys.} 54 (2000), 263-277.

\bibitem{aguiar-infhopf} M. Aguiar, Infinitesimal Hopf algebras, {\em Comtemp. Math.} 267 (2000) 1-30.

\bibitem{aguiar-inf} M. Aguiar, Infinitesimal bialgebras, pre-Lie and dendriform algebras, In ``Hopf algebras", Lecture Notes in Pure and Applied Mathematics, Vol. 237 (2004) 1-33.

\bibitem{agu} M. Aguiar, Dendriform algebras relative to a semigroup, {\em Symmetry, Integrability and Geometry: Methods and Applications SIGMA} 16 (2020), 066, 15 pages.


\bibitem{atkinson} F. V. Atkinson, Some aspects of Baxter's functional equation, {\em J. Math. Anal. Appl.} 7 (1963) 1-30.

\bibitem{B60} G. Baxter, An analytic problem whose solution follows from a simple algebraic identity, {\em  Pacific J. Math.} 10 (1960) 731-742.

\bibitem{BBGN13} C. Bai, O. Bellier, L. Guo and X. Ni, Splitting of operations, Manin products, and Rota-Baxter operators, {\em  Int. Math. Res. Not.} 2013 (2013) 485-524.

\bibitem{bala} D. Balavoine, D\'{e}formations des alg\'{e}bres de Leibniz, {\em C. R. Acad. Sci. Paris S\'{e}r. I Math.} 319 (1994), no. 8, 783-788.

\bibitem{bala1} D. Balavoine, Deformation of algebras over a quadratic operad, {\em Operads: Proceedings of Renaissance Conferences (Hartford, CT/ Luminy, 1995),} 207-234, Contemp. Math., 202, {\em Amer. Math. Soc., Providence, RI}, 1997.


\bibitem{cartier} P. Cartier, On the structure of free Baxter algebras, {\em Adv. Math.} 9 (2) (1972) 253-265.

\bibitem{CK00} A. Connes and D. Kreimer, Renormalization in quantum field theory and the Riemann-Hilbert problem. I. The Hopf algebra structure of graphs and the main theorem, {\em  Comm. Math. Phys.} 210 (2000) 249-273.

\bibitem{Das} A. Das, Deformations of associative Rota-Baxter operators, {\em  J. Algebra} 560 (2020) 144-180.
 


\bibitem{A3} A. Das, Twisted Rota-Baxter operators and Reynolds operators on Lie algebras and NS-Lie algebras, {\em J. Math. Phys.} Vol. 62, Issue 9 (2021) 091701.

\bibitem{A4} A. Das, Cohomology and deformations of dendriform algebras, and Dend$_\infty$-algebras, {\em Comm. Algebra}, DOI: 10.1080/00927872.2021.1985130

\bibitem{An} A. Das, Cohomology and deformations of twisted Rota-Baxter operators and NS-algebras, arXiv preprint arXiv:2010.01156




\bibitem{eb} K. Ebrahimi-Fard, Loday-type algebras and the Rota-Baxter relation, {\em Lett. Math. Phys.} 61 (2002) 139-147.

\bibitem{fard-patras} K. Ebrahimi--Fard, J. Gracia-Bondia and F. Patras, A Lie theoretic approach to renormalization, {\em Comm. Math. Phys.} 276 (2007), 519-549.

\bibitem{foissy} L. Foissy, Typed binary trees and generalized dendriform algebras, {\em J. Algebra} 586 (2021) 1-61.



\bibitem{gers} M. Gerstenhaber, On the deformation of rings and algebras, {\em Ann. of Math.} Vol. 79, No. 1 (1964) 59-103. 

\bibitem{guo09} L. Guo, Operated monoids, Motzkin paths and rooted trees, {\em J. Algebraic Combin.} 29 (2009), 35-62.

\bibitem{guo-book} L. Guo, An Introduction to Rota-Baxter Algebra, Higher Education Press, Beijing, 2012.

\bibitem{guo-keigher} L. Guo and W. Keigher, Baxter algebras and shuffle products, {\em Adv. Math.} 150 (2000) 117-149.

\bibitem{ns-guo} L. Guo and P. Lei, Nijenhuis algebras, NS-algebras and N-dendriform algebras, {\em Frontiers in Math.} 7 (2012) 827-846.

\bibitem{guo-pei} J. Pei and L. Guo, Averaging algebras, Schr\"{o}der numbers, rooted trees and operads, {\em J. Algebr. Comb.} 42 (2015) 73-109.







\bibitem{kre}
D. Kreimer and E. Panzer, Hopf-algebraic renormalization of Kreimer’s toy model, Master thesis, Handbook,
https://arxiv.org/abs/1202.3552

\bibitem{leroux}
P. Leroux, Construction of Nijenhuis operators and dendriform trialgebras, {\em Int. J. Math. Math. Sci.} 49 (2004) 2595-2615.


\bibitem{loday} J.-L. Loday, Dialgebras and related operads. pp. 7--66, Lecture Notes in Math., 1763, {em Springer, Berlin} (2001).

\bibitem{lod-val-book} J.-L. Loday and B. Vallette, Algebraic operads, Springer-Verlag Berlin Heidelberg (2012).




\bibitem{nij-ric} A. Nijenhuis and R. W. Richardson, Cohomology and deformations in graded Lie algebras, {\em Bull. Amer. Math. Soc.} 72 (1966), 1-29.


\bibitem{ospel} C. Ospel, F. Panaite and P. Vanhaecke, Generalized NS-algebras, arXiv:2103.07530

\bibitem{PBG17} J. Pei, C. Bai and L. Guo, Splitting of operads and Rota-Baxter operators on operads, {\em  Appl. Categ. Structures} 25 (2017) 505-538.

\bibitem{reyn} O. Reynolds, On the dynamical theory of incompressible viscous fluids and the determination of the criterion, {\em Phil. Trans. Roy. Soc. A} 136 (1895), 123-164; reprinted in {\em Proc. Roy. Soc. London Ser. A} 451 (1995), no. 1941, 5-47.

\bibitem{rota} G.-C. Rota, Baxter algebras and combinatorial identities I, II. {\em Bull. Amer. Math. Soc.} 75, 325--329 (1969); ibid 75, 330--334 (1969).

\bibitem{sev-wein} P. Severa and A. Weinstein, Poisson geometry with a $3$-form background, {\em Progr. Theoret. Phys. Suppl.} 144 (2001), 145-154.


\bibitem{TBGS19} R. Tang, C. Bai, L. Guo and Y. Sheng, Deformations and their controlling cohomologies of $\mathcal{O}$-operators, {\em  Comm. Math. Phys.} 368 (2019) 665-700.


\bibitem{U08} K. Uchino, Quantum analogy of Poisson geometry, related dendriform algebras and Rota-Baxter operators, {\em  Lett. Math. Phys.} 85 (2008) 91-109.

\bibitem{wang} J. Wang, Y, Zhang, Z. Zhu and D. Chen, Free nonunitary Rota-Baxter family algebras and typed leaf-spaced decorated planar rooted forests, {\em Open Mathematics} Vol. 18, no. 1 (2020) 1173-1184.




\bibitem{zhang-free} Y. Zhang and X. Gao, Free Rota-Baxter family algebras and (tri)dendriform family algebras, {\em Pacific J. Math.} 301 (2019) 741-766. 

\bibitem{zhang} Y. Zhang, X. Gao and D. Manchon, Free (tri)dendriform family algebras, {\em J. Algebra} 547 (2020) 456-493.


\end{thebibliography}
\end{document}